\documentclass[reqno, 12pt]{amsart}
\usepackage[utf8]{inputenc}

\usepackage[english]{babel}

\usepackage{lmodern}
\usepackage[T1]{fontenc}
\usepackage{microtype}
\usepackage[hcentering, width=456pt, lines=46]{geometry}

\usepackage{amsmath}
\usepackage{amsthm}
\usepackage{amsfonts}
\usepackage{amssymb}
\usepackage{mathtools}

\usepackage{enumitem}
\setlist[enumerate,1]{label={(\arabic*)}}
\usepackage{tikz-cd}

\usepackage[backend=biber,style=numeric,sorting=nyt,maxnames=99,doi=false,url=false,eprint=false,isbn=false, sortcites]{biblatex}
\usepackage{xpatch}

\usepackage{hyperref}
\hypersetup{
  colorlinks,
  citecolor=blue,
  filecolor=green,
  linkcolor=darkgray,
  urlcolor=gray
}
\usepackage{csquotes}

\theoremstyle{plain}      
\newtheorem{theorem}{Theorem}[section]
\newtheorem{proposition}[theorem]{Proposition}
\newtheorem{lemma}[theorem]{Lemma}
\newtheorem{corollary}[theorem]{Corollary}

\theoremstyle{definition}
\newtheorem{definition}[theorem]{Definition} 

\newtheorem*{acknowledgments}{Acknowledgments}

\theoremstyle{remark}
\newtheorem*{remark}{Remark}
\newtheorem*{notation}{Notation}

\newcommand*{\N}{\mathbb{N}}

\newcommand*{\C}{\mathbb{C}}
\newcommand*{\I}{\mathbb{I}}
\newcommand*{\T}{\mathbb{T}}
\newcommand*{\cF}{\mathfrak{F}}
\newcommand*{\cI}{\mathcal{I}}
\newcommand*{\cJ}{\mathcal{J}}
\newcommand*{\cT}{\mathcal{T}}
\newcommand*{\cO}{\mathcal{O}}
\newcommand*{\cL}{\mathcal{L}}
\newcommand*{\cK}{\mathcal{K}}

\newcommand*{\cB}{\mathcal{B}}

\newcommand*{\cC}{\mathcal{C}}
\newcommand*{\cW}{\mathcal{W}}
\newcommand*{\cU}{\mathcal{U}}
\newcommand*{\cNT}{\mathcal{N\mkern-2mu T}}
\newcommand*{\cNO}{\mathcal{N\mkern-2mu O}}
\newcommand*{\bA}{\boldsymbol{A}}
\newcommand*{\bX}{\boldsymbol{X}}
\newcommand*{\ba}{\boldsymbol{a}}
\newcommand*{\bb}{\boldsymbol{b}}
\newcommand*{\bx}{\boldsymbol{x}}
\newcommand*{\by}{\boldsymbol{y}}
\newcommand*{\bGamma}{\boldsymbol{\Gamma}}

\newcommand*{\tup}{\mathbf}
\newcommand*{\tupi}[1][i]{\mathbf{1}_{#1}}

\DeclareMathOperator*{\supp}{supp}
\DeclareMathOperator*{\id}{id}
\DeclareMathOperator*{\wlim}{w*-lim}

\newcommand*{\Cs}{$C^*$-algebra}
\newcommand*{\Isep}[1][Y]{\I^s_{#1}}
\newcommand*{\Iinv}[1][Y]{\I^i_{#1}}
\newcommand*{\quot}[1][I]{[{-}]_{#1}}

\author{Boris Bilich}
\title{Ideal structure of Nica-Toeplitz algebras}
\date{\today}
\address{Mathematisches Institut, Georg-August-Universität Göttingen, Bunsenstr. 3–5, 37073 Göttingen, Germany}
\address{Department of Mathematics\\ University of Haifa \\ 199 Aba Khoushy Ave. Mount Carmel, Haifa}
\email{bilichboris1999@gmail.com}


\begin{document}
\begin{abstract}
  We study the gauge-invariant ideal structure of the Nica-Toeplitz algebra $\cNT(X)$ of a product system $(A, X)$ over $\N^n$.
  We obtain a clear description of $X$-invariant ideals in $A$, that is, restrictions of gauge-invariant ideals in $\cNT(X)$ to $A$.
  The main result is a classification of gauge-invariant ideals in $\cNT(X)$ for a proper product system in terms of families of ideals in $A$.
  We also apply our results to higher-rank graphs.
\end{abstract}
\maketitle

\section{Introduction}
Product systems of $C^*$-correspondences over semigroups are a powerful framework for studying symmetries and dynamics in noncommutative geometry.
The theory was first developed by Pimsner \cite{Pims97} for a single correspondence, and then it was extended to discrete semigroups by Fowler \cite{Fow02}.
A product system $(A, X)$ over a \Cs{} $A$ is a collection $\{X^p\}_{p\in P}$ of $A$-$A$-correspondences ($A$-bimodules with special properties) indexed by a semigroup $P$ together with specified isomorphisms $X^p\otimes_A X^q \cong X^{pq}$ (see Definition \ref{d:product_system}).
There are product systems associated to semigroup dynamical systems, higher-rank (topological) graphs \cite{Kat04,KP00,RSY03}, self-similar groups \cite{Nek09, EP17}, and factorial languages \cite{DK2021}, yet this merely represents a portion of the applications of product systems.

Like many other mathematical objects, product systems are studied through their representations (see Definition \ref{d:rep_of_prod_sys}).
There is a universal algebra $\mathcal{T}_X$ for representations of $(A, X)$.
When the semigroup is embedded into a group $G$, the algebra $\cT_X$ comes with a natural coaction of $C^*(G)$ such that $X^p$ is homogeneous of degree $p \in G$.
When the group is abelian, the coaction is equivalent to an action of the dual group $\widehat G$, which is called the \emph{gauge action}.
We only study product systems over $\N^n$, in which case the gauge group is $\T^n$.

However, the algebra $\mathcal{T}_X$ is often too big to be useful.
That is why we want to restrict the set of representations by imposing additional conditions.
The corresponding universal algebra is then a certain gauge-invariant quotient of $\cT_X$.
In case of a single correspondence $X$, Katsura \cite{K2007} classified all the gauge-invariant ideals in $\cT_X$ in terms of pairs of ideals in $A$.
Among these, there is a unique maximal gauge-invariant ideal $\cC_\cI$ which intersects trivially with the base algebra $A$.
The quotient $\cO_X \coloneqq \cT_X/\cC_\cI$ is called the \emph{Cuntz-Pimsner algebra of $X$}.
Katsura also defined \emph{relative Cuntz-Pimsner algebras} as quotients by other gauge-invariant ideals.

The higher rank case is much more complicated since ideals in $\cT_X$ come not only from ideals of $A$ but also from linear relations on projections corresponding to the base semigroup $P$.
Further information about the emerging problems can be found in \cite{Mur96,Li16}.
Based on the work \cite{Nic92} by Nica, Fowler \cite{Fow02} proposed to restrict the set of representations by imposing the so-called Nica-covariance condition.
It turns out that most of the interesting representations are Nica-covariant and in the case of compactly aligned $X$ there is a universal algebra $\cNT(X)$ called the Nica-Toeplitz algebra.
This algebra is much more tractable than $\cT_X$.

A substantial work on the generalization of Katsura's results to the higher rank case was done by several authors.
Sims and Yeend \cite{SY10} defined the Cuntz-Nica-Pimsner algebra $\cNO(X)$ of a product system over a quasi-lattice ordered group.
Their construction unifies algebras of higher-rank graphs and Katsura's Cuntz-Pimsner algebra.
Carlsen, Larsen, Sims, and Vittadello \cite{CLSV11} proposed another definition, which coincides with that of Sims-Yeend under a certain amenability condition.
In terms relevant to our paper, they demonstrated the existence of a unique maximal gauge-invariant ideal within $\cNT(X)$ not meeting the base algebra $A$.
The quotient by this ideal is then the \emph{co-universal} Cuntz-Nica-Pimsner algebra of the product system $(A, X)$.
The co-universal property was extensively studied by several authors in \cite{DK20, DKKLL22, Seh22}.

However, the gauge-invariant ideal structure of the Nica-Toeplitz algebra $\cNT(X)$ remained unknown (see \cite[Question 9.2]{DK2021}).
We aim to fill this gap by describing the gauge-invariant ideal lattice of $\cNT(X)$ in the special case of proper product systems over the semigroups $\N^n$.
The main result is a description of the gauge-invariant ideal lattice of $\cNT(X)$ in Theorem \ref{t:classification-of-ideals}.

In order to classify gauge-invariant ideals in $\cNT(X)$ we adapt the methods of \cite{K2007} to the higher-rank case.
The adaptation significantly relies on the results of Dor-On and Kakariadis \cite{DK2021}.
They defined a wide class of \emph{strongly compactly aligned} product systems and described the Cuntz-Nica-Pimsner algebra $\cNO(X)$ of a strongly compactly aligned $\N^n$-product system $(A, X)$ as a universal algebra for CNP-representations.
These are representations where certain covariance conditions are satisfied on explicitly defined ideals $\{\cI^F\}_{F\subset \{1,\dots, n\}}$ of $A$ (see Definition \ref{d:cnp-rep}).

We first consider an arbitrary strongly compactly aligned $\N^n$-product system $(B, Y)$.
We define \emph{invariant ideals} in $B$ as ideals coming from the gauge-invariant ideals of $\cNO(Y)$ (see Definition \ref{d:inv-ideal}).
In order to describe these ideals we define two weaker notions: positively invariant (Definition \ref{d:pos-inv}) and negatively invariant (Definition \ref{d:neg-inv}) ideals.
Finally, we show that an ideal is invariant if and only if it is negatively and positively invariant (Theorem \ref{t:invariant_if_pos_and_neg}).
The whole Section \ref{s:invariant-ideals} is a direct generalization of \cite[Section 4]{K2007}, including the three notions of invariance.

Next, we focus on a proper product system $(A, X)$ and classify gauge-invariant ideals in $\cNT(X)$ and $\cNO(X)$.
In the rank 1 case, these ideals are classified by T-pairs and O-pairs of ideals in $A$ (see \cite[Section 5]{K2007}).
It turns out that we need families of $2^n$ ideals in the rank $n$ case.
We define three types of families of ideals in $A$: T-families, O-families (see Definition \ref{d:to-fam}) and invariant families (see Definition \ref{d:inv-fam}).
We show that there is a lattice bijection between T-families and invariant families in Proposition \ref{p:T_and_inv_families_equiv}.

Following this, given an invariant family $K = \{ K^F \}_{F\subset \{1,\dots, n\}}$, we define an extended product system $\bX_K$ over a \Cs{} $\bA_K = \bigoplus_F A/K^F$.
We show that $\bX_K$-invariant ideals in $\bA_K$ are in bijection with invariant families $K'$ containing $K$.
Moreover, we show that every invariant ideal of $\bA_K$ is separating, meaning that gauge-invariant ideals of $\cNO(\bX_K)$ are in bijection with invariant ideals of $\bA_K$.
While the construction is inspired by Katsura's \cite[Definition 6.1]{K2007}, it differs even in the rank 1 case.
Unlike Katsura's construction, there is no immediate generalization to the non-proper case.

Finally, for every T-family $I$ we define an \emph{$I$-relative Cuntz-Nica-Pimsner algebra} $\cNO(X, I)$ as a certain gauge-invariant quotient of $\cNT(X)$ (Definition \ref{d:relative-algebra}).
In particular, $\cNO(X, 0)$ is isomorphic to $\cNT(X)$ as expected.
We prove in Proposition \ref{p:iso_of_relative} that $\cNO(X, I)$ is canonically isomorphic to $\cNO(\bX_{K_I})$, where $K_I$ is the invariant family corresponding to $I$.
Together with the classification of invariant ideals of $\bA_K$, this leads to a lattice bijection between gauge-invariant ideals in $\cNT(X)$ (resp. $\cNO(X)$) and T-pairs (resp. O-pairs) in Theorem \ref{t:classification-of-ideals}.
Moreover, this also shows that the quotient by the ideal corresponding to a T-family $I$ is the $I$-relative Cuntz-Nica-Pimsner algebra $\cNO(X, I)$.

The paper is structured as follows.
In Section \ref{s:preliminaries} we give basic definitions and constructions.
Section \ref{s:invariant-ideals} is devoted to the description of invariant ideals.
The main result here is Theorem \ref{t:invariant_if_pos_and_neg}, which states that invariant ideals are exactly positively and negatively invariant ideals.
We start Section \ref{s:gauge-invariant-ideals} with necessary constructions and definitions.
First, the definitions of invariant families, T-families, and O-families are given in Section \ref{ss:families-of-ideals}.
Secondly, we construct an extended product system $(\bA_K, \bX_K)$ in Section \ref{ss:extended-product-system}.
Next, in Definition \ref{d:relative-algebra}, we define $I$-relative algebras and prove the main result of the paper: the classification of gauge-invariant ideals in $\cNT(X)$ and $\cNT(X)$ (Theorem \ref{t:classification-of-ideals}).
Finally, we apply our results to the theory of higher-rank graphs in Section \ref{s:examples}.
We represent any quotient of the Toeplitz algebra of a row-finite higher-rank graph as the algebra of some extended higher rank graph in Theorem \ref{t:graph-classification}, generalizing an analogous result of Bates-Hong-Raeburn-Szyma\'{n}ski \cite[Corollary 3.5]{BHRS02} for rank 1 graphs.

During the preparation of this preprint, the author became aware of independent work by Dessi and Kakariadis \cite{dessi_kakariadis_2023}, who have achieved similar results.

\begin{acknowledgments}
  This work is based on the author's Master's thesis conducted at the University of Göttingen.
  I would like to express my sincere gratitude to my advisor Ralf Meyer for posing the initial problem and providing guidance throughout the preparation of this work.
  Special thanks go to Adam Dor-On for introducing me to the results of \cite{DK2021}, sharing his notes on the topic, and reviewing a draft version of the paper.
  I would also like to extend my appreciation to Chenchang Zhu, who has kindly agreed to undertake the review and evaluation of the Master's thesis.
  I thank Joseph Dessi and Evgenios Kakariadis for sending me their manuscript \cite{dessi_kakariadis_2023}.
\end{acknowledgments}

\section{Preliminaries}
\label{s:preliminaries} In this section, we establish notation and recall basic definitions and constructions as well as relevant recent results.
We refer the reader to \cite{LA95_toolkit} for details on Hilbert $C^*$-modules and tensor products of $C^*$-correspondences.
Fowler's paper \cite{Fow02} is a good reference for product systems and Nica-covariance for quasi-lattice ordered semigroups.
\subsection{Notation}
We denote elements of the semigroup $\N^n$ by bold Latin letters like $\tup m = (\tup m_1, \dots, \tup m_n)$.
If $\tup m, \tup k \in \N^n$, we write $\tup m \leq \tup k$ if $\tup m_i \leq \tup k_i$ for all $i = 1, \dots, n$.
We use the notation $\tup m \vee \tup k$ and $\tup m \wedge \tup k$ for the coordinatewise maximum and minimum of $\tup m$ and $\tup k$, respectively.
These are the meet and the join operations in the lattice $\N^n$ with the partial order $\leq$.

We use $[n]$ to denote the set $\{1, \dots, n\}$ and $\cF \coloneqq 2^{[n]}$ to denote the set of all subsets of $[n]$.
Given $F\in \cF$, we write $\tup 1_F\in \N^n$ for the characteristic vector of $F$: $(\tup 1_F)_i$ is $1$ if $i\in F$ and $0$ otherwise.
We write $\tup 1_i$ instead of $\tup 1_{\{i\}}$ for a singleton $\{i\} \in \cF$ and $\tup 0$ for the zero vector.
Conversely, if $\tup m \in \N^n$, we denote by $\supp \tup m \in \cF$ the subset of non-zero coordinates of $\tup m$.
For $\tup m,\tup k\in \N^n$, we write $\tup m \perp \tup k$ if $\supp \tup m \cap \supp \tup k = \emptyset$. 
\subsection{Hilbert C*-modules}
Let $A$ be a \Cs{}.
A \emph{(right) Hilbert $A$-module} is a right $A$-module $X$ equipped with a map $\langle \cdot, \cdot \rangle \colon X\times X \to A$ which is $\C$-linear in the second variable and such that
\begin{enumerate}
\item $\langle x, y \rangle^* = \langle y, x \rangle$ for all $x,y\in X$,
\item $\langle x, x \rangle \geq 0$ for all $x\in X$,
\item $\langle x, y \rangle \cdot a = \langle x, y\cdot a \rangle$ for all $x,y\in X$ and $a\in A$,
\item $X$ is complete with respect to the norm $\|x\| = \|\langle x, x \rangle\|^{1/2}$.
\end{enumerate}

If $X, Y$ are Hilbert $A$-modules, a bounded operator $T\colon X\to Y$ is called \emph{adjointable} if there exists a bounded operator $T^*\colon Y\to X$ such that $\langle Tx, y \rangle = \langle x, T^*y \rangle$ for all $x\in X$ and $y\in Y$.
The space of all adjointable operators from $X$ to $Y$ is denoted by $\cL_A(X,Y)$.
The space of all adjointable operators from $X$ to itself is a \Cs{} denoted by $\cL_A(X)\coloneqq \cL_A(X,X)$.
One can show that adjointable operators are automatically right $A$-module homomorphisms.

For $x\in X$ and $y\in Y$, we define the \emph{rank-one operator} $\theta_{y, x}\colon X\to Y$ by $\theta_{y, x}(z) = y\langle x, z \rangle$ for all $z\in X$.
The closed linear span of all rank-one operators is denoted by $\cK_A(X,Y)\subset \cL_A(X,Y)$ and called the space of \emph{(generalized) compact operators}.
The \emph{(generalized) compact operators} on $X$ are defined as $\cK_A(X)\coloneqq \cK_A(X,X)$, and they form a closed two-sided ideal in $\cL_A(X)$.
We write $\cK$ and $\cL$ instead of $\cK_A$ and $\cL_A$ if $A$ is clear from the context.

Let $\mathfrak{S}$ be a finite set and let $\{A_S\}_{S\in \mathfrak{S}}$ be a family of $C^*$-algebras.
Let $\bA=\bigoplus_{S\in \mathfrak{S}} A_s$ be their direct sum.
Suppose that we are provided with a Hilbert $A_S$-module $X_S$ for each $S\in \mathfrak{S}$.
Then the direct sum $\bX\coloneqq \bigoplus_{S\in \mathfrak{S}} X_S$ is a Hilbert $\bA$-module with the inner product $\langle \bx, \by \rangle =  (\langle x_S, y_S \rangle)_{S\in \mathfrak S}$ for $\bx = (x_S)_{S\in \mathfrak{S}}$ and $\by = (y_S)_{S\in \mathfrak{S}}$.
\begin{lemma} \label{l:direct_sum}
  Let $\bA = \bigoplus_{S\in \mathfrak{S}} A_S$ and let $\bX$ be a Hilbert $\bA$-module.
  Then it has the form $\bX = \bigoplus_{S\in \mathfrak{S}} X_S$ as above.
  Moreover, we have $\cL(\bX) = \bigoplus_{S\in \mathfrak{S}} \cL(X_S)$ and $\cK(\bX) = \bigoplus_{S\in \mathfrak{S}} \cK(X_S)$.
\end{lemma}
\begin{proof}
  Consider the subspace $X_S = \bX A_S$ of $\bX$.
  It is easy to see that it is a closed submodule and $\langle X_S, X_S \rangle = A_S$.
  Hence, $X_S$ is a Hilbert $A_S$-module.
  We have $\bX = \bX \bA = \bX \bigoplus_{S\in \mathfrak{S}} A_S = \bigoplus_{S\in \mathfrak{S}} \bX A_S = \bigoplus_{S\in \mathfrak S} X_S$, which proves the first claim.

  For the second claim, consider $T\in \cL(\bX)$.
  Then, we have $T X_S = T(\bX A_S) = (T\bX)A_S \subset X_S$, so $T$ can be described as a diagonal operator matrix with diagonal entries $T_S = T|_{X_S}$.
  The same argument works for $\cK(\bX)$.
\end{proof}

\subsection{C*-correspondences}
Let $A$ and $B$ be \Cs{s}.
An \emph{$A$-$B$-correspondence} is a right Hilbert $B$-module $X$ equipped with a $*$-homomorphism $\varphi_X\colon A\to \cL(X)$.
A correspondence is \emph{proper} if $\varphi_X(A) \subset \cK(X)$.
It is \emph{injective} (or faithful) if $\varphi$ is injective and \emph{non-degenerate} if $\varphi_X(A)X$ is dense in $X$.
When the map $\varphi_X(A)$ is clear from the context, we simply write $a\cdot x$ instead of $\varphi_X(a)x$ for $a\in A$ and $x\in X$.

Let $C$ be another \Cs{} and let $Y$ be a $B$-$C$-correspondence.
The \emph{tensor product} $X\otimes_B Y$ is a Hilbert $C$-module defined as the Hausdorff completion of the algebraic tensor product $X\otimes_B Y$ with respect to the inner product
\[
  \langle x_1\otimes y_1, x_2\otimes y_2 \rangle = \langle y_1, \varphi_Y(\langle x_1, x_2 \rangle) y_2 \rangle \in C
\]
for $x_1, x_2\in X$ and $y_1, y_2\in Y$.

There is a map $\iota_{X}^{X\otimes Y}\colon \cL(X) \to \cL(X\otimes_B Y)$, defined by $\iota_{X}^{X\otimes Y}(T)(x\otimes y) = Tx\otimes y$ for $T\in \cL(X)$, $x\in X$ and $y\in Y$.
It is a homomorphism of $C^*$-algebras.
\begin{lemma}[{\cite[Proposition 4.7]{LA95_toolkit}}] \label{l:iota_on_compacts}
  If $Y$ is proper, then $\iota_{X}^{X\otimes Y}(\cK(X)) \subset \cK(X\otimes_B Y)$.
  Moreover, $\iota_{X}^{X\otimes Y}|_{\cK(X)}$ is injective (surjective) if $\varphi_Y$ is injective (surjective).
\end{lemma}

We can now define a structure of $A$-$C$-correspondence on the Hilbert $C$-module $X\otimes_B Y$ by $\varphi_{X\otimes Y} = \iota_{X}^{X\otimes Y}\circ \varphi_X$.
By Lemma \ref{l:iota_on_compacts}, the correspondence $X\otimes_B Y$ is proper if $X$ and $Y$ are proper.

Let $I\subset B$ be an ideal.
Then, we can consider $B/I$ as a $B$-$B/I$-correspondence.
We define the quotient Hilbert $B/I$-module $Y_I \coloneqq Y\otimes_B B/I\cong Y/YI$.
The last isomorphism follows from the fact that $YI$ is a closed submodule of $Y$.
The quotient maps are denoted by $\quot \colon B \to B_I$, $\quot\colon Y \to Y_I$.
For a bounded operator $T\in \cL(Y)$, we abuse notation and write $[T]_I\in \cL(Y_I)$ to denote the operator $\iota_{Y}^{Y_I}(T)$.
\begin{lemma}[{\cite[Lemma 1.6]{K2007}}] \label{l:quot_on_compacts}
  We have
  \[
    [\cK(Y)]_I = \cK(Y_I)
  \]
  and $[T]_I=0$ for compact $T$ if and only if $T\in \cK(YI)$.
\end{lemma}
\begin{proof}
  The $B$-$B/I$-correspondence $B/I$ is proper and the map $\varphi_{B/I}\colon B \to \cK_{B/I}(B/I) = B/I$ is surjective.
  Therefore, by Lemma \ref{l:iota_on_compacts}, the map $\quot = \iota_{B}^{B/I}$ maps $\cK(B)$ onto $\cK(B/I)$. See the proof of \cite[Lemma 1.6]{K2007} for the second claim.
\end{proof}

When $Y$ is an $A$-$B$-correspondence, then $Y_I$ is an $A$-$B/I$-correspondence with left multiplication map $\varphi_{Y_I} = \quot\circ \varphi_Y$.
We define an inclusion-preserving map $Y^{-1}\colon \I(B) \to \I(A)$ from the set of ideals of $B$ to the set of ideals of $A$ by
\[
  Y^{-1}(I) = \ker \varphi_{Y_I} = \{a\in A \colon a\cdot Y \subset YI\}.
\]
\begin{lemma}[{\cite[Proposition 1.3]{K2007}}] \label{l:inv_via_scalar_prod}
  The following are equivalent for an ideal $I\subset B$ and $a\in A$.
  \begin{enumerate}
  \item $a \in Y^{-1}(I)$.
  \item $\langle x, a\cdot y \rangle \in I$ for all $x, y\in Y$.
  \item $\langle x, a\cdot x \rangle \in I$ for all $x\in Y$.
  \end{enumerate}
\end{lemma}
\begin{lemma} \label{l:inv_composition}
  Let $X$ be an $A$-$B$-correspondence and $Y$ be a $B$-$C$-correspondence.
  Then, we have $(X\otimes_B Y)^{-1}(I) = X^{-1}(Y^{-1}(I))\subset A$ for all ideals $I\subset C$.
\end{lemma}
\begin{proof}
  Let $a\in A$, $x_1,x_2\in X$, and $y_1,y_2\in Y$ be arbitrary.
  Then, we have
  \[
    \langle x_1\otimes y_1, a\cdot (x_2\otimes y_2) \rangle = \langle y_1, \langle x_1, a\cdot x_2 \rangle y_2 \rangle.
  \]
  By Lemma \ref{l:inv_via_scalar_prod}, the right-hand side is in $I$ for all $y_1,y_2\in Y$ if and only if $\langle x_1, a\cdot x_2 \rangle\in Y^{-1}(I)$.
  Applying Lemma \ref{l:inv_via_scalar_prod} again, we see that this is equivalent to $a\in X^{-1}(Y^{-1}(I))$.
\end{proof}

Now, suppose that $\boldsymbol{C} = \bigoplus_{S\in \mathfrak{S}} C_S$ is a finite sum of $C^*$-algebras and $\boldsymbol Y = \bigoplus_{S\in \mathfrak{S}} Y_S$ is a Hilbert $\boldsymbol{C}$-module.
By Lemma \ref{l:direct_sum}, a homomorphism $\varphi_{\boldsymbol Y}\colon B\to \cL(\boldsymbol Y)$ is equivalent to a family of homomorphisms $\varphi_{Y_S}\colon B\to \cL(Y_S)$ for all $S\in \mathfrak{S}$.
Therefore, any $B$-$C$-correspondence is isomorphic to a direct sum of a family of $B$-$C_S$-correspondences.
\begin{lemma} \label{l:tensor_over_sum}
  Let $\boldsymbol Y$ be as above and let $X$ be an $A$-$B$-correspondence.
  Then, the tensor product of $X$ and $\boldsymbol Y$ is given by
  \[
    X\otimes_B \boldsymbol Y = \bigoplus_{S\in \mathfrak{S}} X\otimes_B Y_S.
  \]
\end{lemma}
\begin{proof}
  Observe that the decomposition $\boldsymbol Y = \bigoplus_{S\in \mathfrak{S}} Y_S$ is a direct sum of left $B$-modules.
  The statement follows because the tensor product is distributive over direct sums.
\end{proof}

\subsection{Product systems}
\begin{definition} \label{d:product_system}
  An \emph{$\N^n$-product system} $(A, X)$ over a \Cs{} $A$ is a collection of $A$-$A$-correspondences $X = (X^{\tup m})_{\tup m \in \N^n}$ such that $X^{\tup 0} = A$ together with isomorphisms of correspondences
  \[
    \mu^{\tup m, \tup k}_X\colon X^{\tup m}\otimes_A X^{\tup k} \to X^{\tup m + \tup k}
  \]
  for all $\tup m, \tup k \in \N^n\setminus \{\tup 0\}$, satisfying certain associativity conditions.
  A product system is called proper, injective, or non-degenerate if all its correspondences are.
\end{definition}
We denote the left action homomorphism $A\to \cL(X^{\tup m})$ by $\varphi_X^{\tup m}$ for $\tup m\in \N^n\setminus \{\tup 0\}$.
We also write $x\cdot y$ instead of $\mu_X^{\tup m, \tup k}(x\otimes y)$ for $x\in X^{\tup m}$ and $y\in X^{\tup k}$ for $\tup m, \tup k\in \N^n\setminus \{\tup 0\}$.

For $\tup m \geq \tup k$, the conjugation of the map $\iota_{X^{\tup k}}^{X^{\tup k}\otimes X^{\tup m - \tup k}}$ with the multiplication map $\mu_X^{\tup k, \tup m - \tup k}$ gives a map $\iota_{\tup k}^{\tup m}\colon \cL(X^{\tup k})\to \cL(X^\tup m)$.
It is given by $\iota_{\tup k}^{\tup m}(T)(x\cdot y) = (Tx)\cdot y$ for all $T\in \cL(X^{\tup k})$, $x\in X^{\tup k}$ and $y\in X^{\tup m - \tup k}$.
A product system is called \emph{compactly aligned} if $S\vee T \coloneqq \iota_{\tup k}^{\tup k\vee \tup m}(S)\iota_{\tup m}^{\tup k\vee \tup m}(T) \in \cK(X^{\tup k \vee \tup m})$ for all $S\in \cK(X^{\tup k})$ and $T\in \cK(X^{\tup m})$.
If, additionally, $\iota_{\tup m}^{\tup m + \tupi}(\cK(X^{\tup m}))\subset \cK(X^{\tup m + \tupi})$ for all $i\notin\supp \tup m$ and $\tup m \in \N^n\setminus\{\tup 0\}$, we say that the product system is \emph{strongly compactly aligned} (see \cite[Definition 2.2]{DK2021}).

\begin{definition} \label{d:rep_of_prod_sys}
  A \emph{representation} of a product system $(A, X)$ in a \Cs{} $D$ is a pair $(\sigma, s)$ consisting of a $*$-homomorphism $\sigma\colon A\to D$ and a family of maps $s = (s^{\tup m} \colon X^{\tup m} \to D)_{\tup m \in \N^n\setminus \{\tup 0\}}$ such that the following conditions are satisfied for all $\tup m, \tup k \in \N^n\setminus \{\tup 0\}$:
  \begin{enumerate}
  \item $s^{\tup m}(x)^*s^{\tup m}(y) = \sigma(\langle x, y\rangle)$ for all $x, y\in X^{\tup m}$,
  \item $s^{\tup m}(x)s^{\tup k}(y) = s^{\tup m + \tup k}(x\cdot y)$ for all $x\in X^{\tup m}$, $y\in X^{\tup k}$,
  \item $\sigma(a)s^{\tup m}(x) = s^{\tup m}(\varphi^{\tup m}(a) x)$ for all $a\in A$, $x\in X^{\tup m}$.
  \end{enumerate}
  A representation $(\sigma, s)$ \emph{admits a gauge action} if there is a pointwise norm-continuous action $\gamma\colon \T^n \curvearrowright D$ such that it fixes $\sigma(A)$ and acts by the character $\tup z \mapsto \tup{z}^{\tup m}$ on $s^{\tup m}(X^{\tup m})$ for all $\tup m \in \N^n\setminus \{\tup 0\}$.
  That is, for all $x\in X^{\tup m}$ and $\tup z = (z_1, \dots, z_n)\in \T^n$, we have $\gamma_{\tup z}(s^{\tup m}(x)) = z_1^{\tup m_1}\cdot \dots \cdot z_n^{\tup m_n}\cdot s^{\tup m}(x)$.
\end{definition}
When the degree $\tup m$ of $x\in X^{\tup m}$ is clear, we sometimes write $s(x)$ instead of $s^{\tup m}(x)$.

Given a representation $(\sigma, s)$, we define homomorphisms $\psi_{s}^{\tup m}\colon \cK(X^{\tup m}) \to D$ for all $\tup m \in \N^n\setminus \{\tup 0\}$ first on rank-one operators by
\[
  \psi_{s}^{\tup m}(\theta_{x, y}) = s^{\tup m}(x)s^{\tup m}(y)^*
\]
for $x, y\in X^{\tup m}$.
Then, we may extend $\psi_{s}^{\tup m}$ to all of $\cK(X^{\tup m})$ by linearity and continuity (see \cite[Lemma 2.2]{KPW98}).

Let $(k^{\tup m}_\lambda)_{\lambda \in \Lambda}$ be a canonical approximate identity for $\cK(X^{\tup m})$.
We define projections
\[
  p_{s}^{\tup m} = \wlim_{\lambda\in \Lambda} \psi_{s}^{\tup m}(k^{\tup m}_\lambda) \in D'',
\]
where $D''$ is the enveloping von Neumann algebra of $D$, i.e., the strong closure of $D$ in its universal representation (see \cite[3.7.6]{Pedersen18}).

\begin{definition}[{\cite[Definition 5.1]{Fow02}}] \label{d:Nica_covariance}
  A representation $(\sigma, s)$ of a product system $(A, X)$ on $D$ is called Nica-covariant if $p_s^{\tup m}p_s^{\tup k} = p_s^{\tup m \vee \tup k}$ for all $\tup m, \tup k \in \N^n\setminus \{\tup 0\}$.
\end{definition}
\begin{remark}
  Fowler defines Nica-covariance for representations of non-degenerate product systems on Hilbert spaces.
  Let $\pi_D\colon D\to \cB(H)$ be the universal representation of $D$.
  Since we defined the projections $p^{\tup m}_s$ as elements of the enveloping von Neumann algebra $D''\subset \cB(H)$, our definition is equivalent to the Nica-covariance (in the sense of Fowler) of the representation $(\pi_D\circ \sigma, \pi_D\circ s)$ of $(A, X)$ on the Hilbert space $H$.
  Moreover, Fowler shows in \cite[Proposition 5.6]{Fow02} that for non-degenerate compactly-aligned product systems, the Nica-covariance is equivalent to $\psi_s^{\tup m}(S)\psi_s^{\tup k}(T) = \psi_s^{\tup m\vee \tup k}(S\vee T)$ for all nonzero $\tup m, \tup k \in \N^n$ and for all $S\in \cK(X^{\tup m})$, $T\in \cK(X^{\tup k})$.
  However, he never uses the non-degeneracy assumption in the proof, so this fact holds for all compactly aligned product systems and applies to our definition.
  For the sake of completeness, we include a proof of this fact here.
\end{remark}
\begin{proposition}[{\cite[Proposition 5.6]{Fow02}}]
  A representation $(\sigma, s)$ is Nica-covariant in the sense of Definition \ref{d:Nica_covariance} if and only if $\psi_s^{\tup m}(S)\psi_s^{\tup k}(T) = \psi_s^{\tup m\vee \tup k}(S\vee T)$ for all nonzero $\tup m, \tup k \in \N^n$ and for all $S\in \cK(X^{\tup m})$, $T\in \cK(X^{\tup k})$.
\end{proposition}
\begin{proof}
  Suppose that $(\sigma, s)$ is Nica-covariant.
  For $S$ and $T$ as in the statement, we have
  \begin{multline*}
    \psi^{\tup m}_s(S)\psi^{\tup k}_s(T) = \psi^{\tup m}_s(S)p^{\tup m}_s p^{\tup k}_s \psi^{\tup k}_s(T) = \psi^{\tup m}_s(S)p^{\tup m \vee \tup k}_s \psi^{\tup k}_s(T) = \\\wlim_{\lambda\in \Lambda} \psi^{\tup m}_s(S)\psi^{\tup m \vee \tup k} (k^{\tup m \vee \tup k}_\lambda) \psi^{\tup k}_s(T) = \\\wlim_{\lambda\in \Lambda} \psi^{\tup m \vee \tup k}_s(\iota_{\tup m}^{\tup m \vee \tup k}(S) k^{\tup m \vee \tup k}_\lambda \iota_{\tup k}^{\tup m \vee \tup k}(T)) = \psi^{\tup m \vee \tup k}_s(S\vee T).
  \end{multline*}
  This proves the ``only if'' direction.

  Conversely, suppose that $\psi_s^{\tup m}(S)\psi_s^{\tup k}(T) = \psi_s^{\tup m\vee \tup k}(S\vee T)$ for all nonzero $\tup m, \tup k \in \N^n$ and for all $S\in \cK(X^{\tup m})$, $T\in \cK(X^{\tup k})$.
  We have
  \[
    p^{\tup m}_s p^{\tup k}_s = \wlim_{\lambda\in \Lambda} \psi_s^{\tup m}(k^{\tup m}_{\lambda})\psi_s^{\tup k}(k^{\tup k}_{\lambda}) = \wlim_{\lambda\in \Lambda} \psi_s^{\tup m\vee \tup k}(k^{\tup m}_{\lambda} \vee k^{\tup k}_{\lambda}),
  \]
  where in the first equality we have used that the multiplication is jointly strongly continuous on bounded sets by \cite[I.3.2.1]{Blackadar06}.
  It is easy to see that $k^{\tup m}_{\lambda} \vee k^{\tup k}_{\lambda}$ is an approximate identity for $\cK(X^{\tup m \vee \tup k})$, so we conclude that $p^{\tup m}_s p^{\tup k}_s = p^{\tup m \vee \tup k}_s$ and the representation is Nica-covariant.
\end{proof}
The \emph{Nica-Toeplitz algebra} $\cNT(X)$ is the universal algebra generated by $A$ and $X$ with respect to Nica-covariant representations.
Fowler proved its existence in \cite[Theorem 6.3]{Fow02}.
The representation $(\tau_X, t_X)$ of $(A, X)$ on $\cNT(X)$ is then the universal Nica-covariant representation: if $(\sigma, s)$ is another Nica-covariant representation, then there is a unique homomorphism $\sigma \times_0 s \colon \cNT(X) \to D$ such that $(\sigma, s) = ((\sigma \times_0 s)\circ \tau_X, (\sigma \times_0 s)\circ t_X)$.

We introduce two more families of projections.
For $F\in \cF$, we define
\begin{align*}
  Q_s^F & = \prod_{i\in F} (1 - p_s^{\tupi}),                                                                    \\
  P_s^F & = \prod_{i\in F} (1 - p_s^{\tupi})\prod_{i\notin F} p_s^{\tupi} = Q_s^F \prod_{i\notin F} p_s^{\tupi}.
\end{align*}
Here, we mean $Q_s^{\emptyset} = 1$ by the empty product.

\begin{lemma} \label{l:prop_of_projs}
  Let $(\sigma, s)$ be a Nica-covariant representation of a product system $(A, X)$ on $D$.
  Then, the projections introduced above have the following properties.
  Let $m\in \N^n\setminus \{\tup 0\}$ and $F\in \cF$ be arbitrary.
  \begin{enumerate}
  \item Elements of $\sigma(A)$ commute with $p^{\tup m}_s$, $Q^{F}_s$ and $P^{F}_s$.
    Additionally, for all $a\in (\varphi^{\tup m})^{-1}(\cK(X^{\tup m}))$, we have $\sigma(a)p_s^{\tup m} = p_s^{\tup m}\sigma(a) = \psi_s^{\tup m}(\varphi^{\tup m}(a)) \in D$.
    \label{l:prop_of_projs:p_of_a}
  \item The equality $p^{\tup m}_s s^{\tup m}(x) = s^{\tup m}(x)$ holds for all $x\in X^{\tup m}$.
    If $p\in D''$ is some other projection with this property, then $p\geq p^{\tup m}_s$. \label{l:prop_of_projs:smallest}
  \item For all $x\in X^{\tup m}$, we have $p^{\tup 1_i}_s s^{\tup m}(x) = s^{\tup m}(x)p^{\tup 1_i}_s$ if $i\not\in \supp \tup m$ and $p^{\tup 1_i}_s s^{\tup m}(x) = s^{\tup m}(x)$ otherwise. \label{l:prop_of_projs:ps_sp}
  \item For all $F,G\in \cF$ and $x\in X^{\tup m}$, we have \label{l:prop_of_projs:PFsPG}
    \[
      P^F_s s^{\tup m}(x) P^G_s =
      \begin{cases}
        s^{\tup m}(x)P^G_s & \text{if } F = G\setminus \supp \tup m, \\
        0                  & \text{otherwise}.
      \end{cases}
    \]
  \item For a homomorphism $f\colon D\to D_2$, we have $f(p_s^{\tup m}) = p_{f\circ s}^{\tup m}$, $f(Q_s^{F}) = Q_{f\circ s}^{F}$, and $f(P_s^{F}) = P_{f\circ s}^{F}$.
    In particular, $f\circ (\sigma, s)$ is a Nica-covariant representation. \label{l:prop_of_projs:composition}
  \end{enumerate}
\end{lemma}
\begin{proof}
  To prove \ref{l:prop_of_projs:p_of_a}, we may assume that $A$ is unital so that it is a linear span of unitary elements.
  Consider an arbitrary unitary element $u\in A$.
  Then, we have
  \[
    \begin{split}
      \sigma(u)p^{\tup m}_s = \sigma(u)p^{\tup m}_s \sigma(u)^*\sigma(u) &=  \wlim_{\lambda\in \Lambda} \sigma(u) \psi^{\tup m}_s(k_\lambda^{\tup m})\sigma(u)^*\sigma(u) = \\ &= \wlim_{\lambda\in \Lambda} \psi^{\tup m}_s(\varphi^{\tup m}(u)k_\lambda^{\tup m} \varphi^{\tup m}(u^*))\sigma(u).
    \end{split}
  \]
  Since $u$ is unitary, the net $\varphi^{\tup m}(u)k_\lambda^{\tup m} \varphi^{\tup m}(u^*)$ is also a c.a.i. for $\cK(X^{\tup m})$.
  Therefore, the right-hand side of the above equation equals $p^{\tup m}_s \sigma(u)$.
  We conclude that $\sigma(A)$ commutes with $p^{\tup m}_s$ and hence with $Q^F_s$ and $P^F_s$.

  Let $a$ be an element of $(\varphi^{\tup m})^{-1}(\cK(X^{\tup m}))$.
  Then, we have
  \[
    \sigma(a)p_s^{\tup m} = \wlim_{\lambda\in \Lambda} \sigma(a) \psi^{\tup m}_s(k_\lambda^{\tup m}) = \wlim_{\lambda\in \Lambda} \psi^{\tup m}_s(\varphi^{\tup m}(a)k_\lambda^{\tup m}) = \psi^{\tup m}_s(\varphi^{\tup m}(a)),
  \]
  where the last equality follows from the definition of approximate identity.
  This is the second part of \ref{l:prop_of_projs:p_of_a}.

  To prove \ref{l:prop_of_projs:smallest}, consider an arbitrary $x\in X^{\tup m}$.
  Then, we have
  \[
    p^{\tup m}_s s^{\tup m}(x) = \wlim_{\lambda\in \Lambda} \psi^{\tup m}_s(k_\lambda^{\tup m}) s^{\tup m}(x) = \wlim_{\lambda\in \Lambda} s^{\tup m}(k_\lambda^{\tup m}x) = s^{\tup m}(x).
  \]
  Suppose that $p$ is some other projection with this property.
  Observe that each $\psi^{\tup m}_s(k_\lambda^{\tup m})$ lies in the closed linear span of elements of the form $s(x)s(y)^*$ for $x,y\in X^{\tup m}$.
  Therefore, we have $p\psi^{\tup m}_s(k_\lambda^{\tup m}) = \psi^{\tup m}_s(k_\lambda^{\tup m})$ and hence $p p^{\tup m}_s = p^{\tup m}_s$.
  The latter means $p\geq p^{\tup m}_s$ by definition.

  Let us now prove \ref{l:prop_of_projs:ps_sp}.
  We have $p^{\tup 1_i}_s s^{\tup m}(x) = p^{\tup 1_i}_s p^{\tup m}_s s^{\tup m}(x) = p^{\tup 1_i \vee \tup m}_s s^{\tup m}(x)$.
  If $i\in \supp \tup m$, then $\tup 1_i \vee \tup m = \tup m$ and the above equation gives $p^{\tup 1_i}_s s^{\tup m}(x) = s^{\tup m}(x)$.
  If $i\notin \supp \tup m$, then $\tup 1_i \vee \tup m = \tup m + \tup 1_i$.
  Let $y,z\in X^{\tup 1_i}$ be arbitrary.
  Then, we have
  \begin{multline*}
    p^{\tup m + \tup 1_i}_s s^{\tup m}(x) \psi_s^{\tup 1_i}(\theta_{y,z}) = p^{\tup m + \tup 1_i}_s s^{\tup m}(x) s^{\tup 1_i}(y)s^{\tup 1_i}(z)^* = p^{\tup m + \tup 1_i} s^{\tup m + \tup 1_i}(x\cdot y) s^{\tup 1_i}(z)^*\\ = s^{\tup m + \tup 1_i}(x\cdot y)s^{\tup 1_i}(z)^* = s^{\tup m}(x)\psi^{\tup 1_i}_s(\theta_{y,z}).
  \end{multline*}
  By linearity and continuity, it follows that $p_s^{\tup 1_i}s^{\tup m}(x)\psi^{\tup 1_i}_s(T) = s^{\tup m}(x)\psi^{\tup 1_i}_s(T)$ for all $T\in \cK(X^{\tup 1_i})$.
  Therefore, we also have $p_s^{\tup 1_i}s^{\tup m}(x)p_s^{\tup 1_i} = s^{\tup m}(x)p_s^{\tup 1_i}$ by applying the above to $T = k_\lambda^{\tup 1_i}$ and taking the limit.

  Now, consider arbitrary elements $y,z\in X^{\tup 1_i}$ and $y', z'\in X^{\tup m}$.
  We have $y'\cdot y, z'\cdot z \in X^{\tup m + \tup 1_i}$ and hence
  \begin{multline*}
    \psi^{\tup m + \tup 1_i}_s(\theta_{y'\cdot y, z'\cdot z})s(x) = s(y'\cdot y)s(z'\cdot z)^*s(x) = s(y')s(y)s(z)^*s(z')^*s(x)\\
    = s(y')\psi^{\tup 1_i}_s(\theta_{y, z})\sigma(\langle z', x\rangle) = s(y')\psi^{\tup 1_i}_s(\theta_{y, z}\varphi^{\tup 1_i}(\langle z', x\rangle))\in s(y')\psi^{\tup 1_i}_s(\cK(X^{\tup 1_i})).
  \end{multline*}
  Since $p^{\tupi}_s$ fixes $\psi_s^{\tupi}(\cK(X^{\tupi}))$ with right multiplication, the inclusion implies that for any $T\in \cK(X^{\tup m + \tup 1_i})$ we have $\psi^{\tup m + \tup 1_i}_s(T)s(x)p^{\tup 1_i}_s = \psi^{\tup m + \tup 1_i}_s(T)s(x)$.
  Again, we get $p^{\tup 1_i}_s s(x) p^{\tup 1_i}_s = p^{\tup m + \tup 1_i}_s s(x) p^{\tup 1_i}_s = s(x)p^{\tup 1_i}_s$ by applying the above to $T = k_\lambda^{\tup m + \tup 1_i}$ and taking the limit.
  This, together with the previous paragraph, gives $p^{\tup 1_i}_s s^{\tup m}(x) = s^{\tup m}(x)$.

  We can prove \ref{l:prop_of_projs:PFsPG} by expanding
  \[
    P^F_s s^{\tup m}(x) P^G_s = \prod_{i\in F} (1 - p^{\tup 1_i}_s) \prod_{i\notin F} p^{\tup 1_i}_s s^{\tup m}(x) \prod_{i\in G} (1 - p^{\tup 1_i}_s) \prod_{i\notin G} p^{\tup 1_i}_s.
  \]
  If there is $i\in F \cap \supp \tup m$, then the factor $1 - p^{\tup 1_i}_s$ annihilates $s^{\tup m}(x)$ by \ref{l:prop_of_projs:ps_sp} and the above expression is zero.
  Otherwise, we use \ref{l:prop_of_projs:ps_sp} iteratively to get
  \[
    P^F_s s^{\tup m}(x) P^G_s = s^{\tup m}(x) \prod_{i\in F}(1-p^{\tup 1_i}_s)\prod_{i\in G}(1-p^{\tup 1_i}_s)\prod_{i\notin \supp \tup m\cup F} p^{\tup 1_i}_s \prod_{i\notin G} p^{\tup 1_i}_s.
  \]
  This is nonzero if and only if $F\subset G$ and $G \subset \supp \tup m \cup F$ or, equivalently, $F = G\setminus \supp \tup m$.
  When $F = G\setminus \supp \tup m$, this equals $s^{\tup m}(x)P^G_s$.

  Finally, \ref{l:prop_of_projs:composition} follows easily from the fact that $\psi^{\tup m}_{f\circ s} = f\circ \psi^{\tup m}_s$.
  This fact is trivial for rank-one operators and hence holds in general.
\end{proof}

We now recall the results of Dor-On and Kakariadis \cite{DK2021}.
Suppose that $(A, X)$ is a strongly compactly aligned product system.
The \emph{CNP-ideals} $\cI_X^F$ are defined in two steps.
First, we define the family of \emph{pre-CNP-ideals} for $F\in \cF$ by
\begin{align} \label{e:def_of_j}
  \cJ_X^F & = (\bigcap_{i\in F}\ker \varphi^{\tupi})^\perp \cap \bigcap_{i = 1}^n (\varphi^{\tupi})^{-1}(\cK(X^{\tupi})) \subset A, \\
  \shortintertext{and then we set}
  \label{e:def_of_i}
  \cI_X^F & = \cJ_X^F \cap \bigcap_{\tup m\perp \tup 1_F} (X^{\tup m})^{-1}(\cJ_X^F).
\end{align}
\begin{definition}[{\cite[Definition 2.8]{DK2021}}]\label{d:cnp-rep}
  A Nica-covariant representation $(\sigma, s)$ of $(A, X)$ is \emph{Cuntz-Nica-Pimsner} (CNP-representation) if
  \[
    \sigma(a)\cdot Q_s^F = \sum_{\tup 0 \leq \tup m \leq \tup 1_F} (-1)^{|\tup m|}\psi_s^{\tup m}(\varphi^{\tup m}(a)) =  0 \text{ for all } F\in \cF \text{ and } a\in\cI_X^F.
  \]
\end{definition}
By construction, we have $\cI^F_X \subset \bigcap_{i=1}^n (\varphi^{\tupi})^{-1}(\cK(X^{\tupi})) = \bigcap_{F\in \cF}(\varphi^{\tup 1_F})^{-1}(\cK(X^{\tup 1_F}))$, where the last equality follows from strong compact alignment.
Therefore, the inclusion $\sigma(\cI_X^F)\cdot Q^F_t \subset \cNT(X)$ holds by Lemma \ref{l:prop_of_projs}.\ref{l:prop_of_projs:p_of_a} and these subspaces generate a gauge-invariant ideal $\cC_{\cI_X}$ in $\cNT(X)$.
A Nica-covariant representation $(\sigma, s)$ is CNP if and only if $\sigma\times_0 s$ factors through $\cNT(X)/\cC_{\cI_X}$.
We use the notation $(\omega_X, o_X)$ for the representation of $(A, X)$ on the \emph{Cuntz-Nica-Pimsner} algebra $\cNO(X) \coloneqq \cNT(X)/\cC_{\cI_X}$.
By the above, it is universal with respect to CNP-representations.
If $(\sigma, s)$ is a CNP-representation on $D$, then we denote by $\sigma\times s$ the induced map $\cNO(X)\to D$.

In our proofs, we need the following lemma.
\begin{lemma} \label{l:cIF_is_max}
  Let $(\sigma, s)$ be a Nica-covariant representation of $(A, X)$ such that $\sigma$ is injective.
  Suppose that $\sigma(a)Q_s^F = 0$ holds for some $a\in \bigcap_{i = 1}^n (\varphi^{\tupi})^{-1}(\cK(X^{\tupi})) \subset A$ and $F\in \cF$.
  Then, $a$ is an element of $\cI^F$.
\end{lemma}
\begin{proof}
  The proof of \cite[Proposition 3.4]{DK2021} uses only the assumptions on $a$ from the lemma.
  Therefore, it applies to this situation and we conclude that $a\in \cI^F$.
\end{proof}

\begin{proposition}[Gauge-invariant uniqueness theorem, {\cite[Theorem 4.2]{DK2021}}] \label{p:giut}
  Suppose that $(\sigma, s)$ is a CNP-representation of $(A, X)$.
  The map $\sigma\times s$ is a faithful representation of $\cNO(X)$ if and only if $\sigma$ is faithful and $(\sigma, s)$ admits a gauge action.
\end{proposition}
\begin{proposition}[Co-universal property, {\cite[Corollary 4.7]{DK2021}}] \label{p:co-universal}
  Let $(\sigma, s)$ be a Nica-covariant representation of $(A, X)$ on $D$ such that it admits a gauge action, $\sigma$ is faithful, and $\sigma\times_0 s$ is surjective.
  Then, there is a unique surjective homomorphism $\Omega_{\sigma,s}\colon D \to \cNO(X)$ such that $(\omega_X, o_X) = \Omega_{\sigma,s}\circ (\sigma, s)$.
\end{proposition}
The co-universal property was proven in increasing levels of generality for product systems over more general classes of semigroups in \cite{DK20,DKKLL22,Seh22}.


\section{Invariant ideals}
\label{s:invariant-ideals}
In this section we determine ideals of a base algebra coming from the CNP-algebra.
All the definitions and results of this section are a direct generalization of Katsura's results in \cite[Section 4]{K2007}.

Consider a strongly compactly aligned $\N^n$-product system $(B,Y)$.
Let $(\omega, o)$ be the representation of $(B, Y)$ on the CNP-algebra $\cNO(Y)$.
We use $\I(B)$ to denote the set of all ideals of $B$, and $\I^{\gamma}(\cNO(Y))$ to denote the set of all gauge-invariant ideals of $\cNO(Y)$.
We define the restriction ${-}^r\colon \I^{\gamma}(\cNO(Y)) \to \I(B)$ and induction ${-}^i \colon \I(B) \to \I^{\gamma}(\cNO(Y))$ maps by
\[
  J^r = \omega^{-1}(J) \quad \text{and} \quad I^i = \overline{\cNO(Y)\omega(I)\cNO(Y)}
\]
for $J \in \I^{\gamma}(\cNO(Y))$ and $I \in \I(B)$ (see \cite[Section 3]{Green78} for more properties of these maps).
\begin{definition} \label{d:inv-ideal}
  An ideal $I \in \I(B)$ is said to be \emph{$Y$-invariant} if it can be expressed as $I = J^r$ for some $J \in \I^{\gamma}(\cNO(Y))$.
  If such an ideal $J$ is unique, $I$ is called \emph{$Y$-separating}.
  When the context is clear, we simply use the terms \emph{invariant} and \emph{separating} instead of $Y$-invariant and $Y$-separating.
  We denote the sets of separating and invariant ideals by $\Isep(B)\subset \Iinv(B) \subset \I(B)$.
\end{definition}
Our goal is to characterize the sets $\Isep(B)$ and $\Iinv(B)$.

\begin{definition} \label{d:pos-inv}
  An ideal $I\subset B$ is called \emph{positively $Y$-invariant} if $IY^{\tupi} \subset Y^{\tupi}I$ for all $i\in [n]$.
\end{definition}
From positive invariance, it automatically follows that $IY^{\tup m} \subset Y^{\tup m}I$ for all $\tup m\in \N^n$.
It is easy to see that in the rank one case, Definition \ref{d:pos-inv} coincides with Katsura's definition of positive invariance \cite[Definition 4.8]{K2007}.

\begin{proposition}
  An invariant ideal $I\in \Iinv(B)$ is positively invariant.
\end{proposition}
\begin{proof}
  Let $J \in \I^{\gamma}(\cNO(Y))$ be such that $I = J^r$.
  Consider the quotient representation $(\sigma, s) = (\quot[J]\circ\omega, \quot[J]\circ o)$ of $(B, Y)$ on $\cNO(Y)/J$.
  Observe that the kernel of $\sigma$ is $I$.
  Then, for all $i\in [n]$, $x,y\in Y^{\tupi}$, and $b \in I$, we have
  \[
    0 = s(x)^* \sigma(b) s(y) = \sigma(\langle x, b\cdot y\rangle).
  \]
  Hence, we have $\langle x, b\cdot y\rangle \in I$.
  We obtain $IY^{\tupi} \subset Y^{\tupi}I$ by Lemma \ref{l:inv_via_scalar_prod}.
  We conclude that $I$ is positively invariant.
\end{proof}

\begin{proposition}
  Let $I\subset B$ be a positively invariant ideal.
  The left action of $B$ on $Y^{\tup m}$ descends to an action of $B_I = B/I$ on $Y^{\tup m}_I$.
  This turns $Y_I = (Y_I^{\tup m})_{\tup m \in \N^n}$ into a strongly compactly aligned $\N^n$-product system over $B_I$.
\end{proposition}
\begin{proof}
  Consider an arbitrary element $b\in I$.
  Then, $\varphi_Y^{\tup m}(b) Y^{\tup m} \subset Y^{\tup m}I$ for all $\tup m\in \N^n$.
  Hence, we have $[\varphi_Y^{\tup m}(b)]_I = 0$ by Lemma \ref{l:quot_on_compacts}.
  Therefore, $I$ is in the kernel of the map $\quot\circ \varphi_Y^{\tup m} \colon B \to \cL(Y_I^{\tup m})$, so that it descends to a map $\varphi_{Y_I}^{\tup m} \colon B_I \to \cL(Y_I^{\tup m})$.
  This proves the first claim.

  We now show that $Y_I$ is a product system.
  For any $\tup m, \tup k \in \N^n\setminus\{\tup 0\}$, where a unitary multiplication map $\mu =\mu^{\tup m, \tup k}_Y \colon Y^{\tup m} \otimes_B Y^{\tup k} \to Y^{\tup m + \tup k}$.
  The map $\mu \otimes \id_{B_I}\colon Y^{\tup m} \otimes_B Y^{\tup k}_{I} = Y^{\tup m} \otimes_B Y^{\tup k}\otimes_B B_I \to Y^{\tup m + \tup k}_I$ is also unitary.
  Moreover, we have $Y^{\tup m}I \otimes_B Y^{\tup k} = Y^{\tup m} \otimes_B I Y^{\tup k} \subset Y^{\tup m} \otimes_B Y^{\tup k}I$, so that $Y^{\tup m}I \otimes_{B} Y^{\tup k} \otimes_B B_I = 0$ and hence $Y^{\tup m} \otimes_B Y^{\tup k}_I \cong Y^{\tup m}_I \otimes_{B_I} Y^{\tup k}_I$ as $B$-$B_I$-correspondences.
  Therefore, $\mu \otimes \id_{B_I}$ defines a unitary multiplication map $\mu_{Y_I}^{\tup m, \tup k}\colon Y_I^{\tup m} \otimes_{B_I} Y_I^{\tup k} \to Y_I^{\tup m + \tup k}$.
  This map is given by $\mu_{Y_I}^{\tup m, \tup k}([x]_I \otimes [y]_I) = [\mu(x\otimes y)]_I$.
  The associativity of the multiplication is trivial from this formula and the associativity of $\mu$.

  We now use Lemma \ref{l:quot_on_compacts} to show that $Y_I$ is strongly compactly aligned.
  Let $S\in \cK(Y_I^{\tup m})$ and $T\in \cK(Y_I^{\tup k})$.
  Then, we may find $S'\in \cK(Y^{\tup m})$ and $T'\in \cK(Y^{\tup k})$ such that $S = [S']_I$ and $T = [T']_I$.
  Then, we have $S\vee T = [S'\vee T']_I \in \cK(Y_I^{\tup m \vee \tup k})$.
  This proves that $Y_I$ is compactly aligned.

  Analogously, for strong compact alignment we consider $T\in \cK(Y_I^{\tup m})$ and choose its preimage $T' \in \cK(Y^{\tup m})$.
  Then, for all $i\notin \supp\tup m$, we have $\iota_{\tup m}^{\tup m + \tupi}(T)= [\iota_{\tup m}^{\tup m + \tupi}(T')]_I \in \cK(Y_I^{\tup m + \tupi})$. We conclude that $Y_I$ is a strongly compactly aligned $\N^n$-product system over $B_I$.
\end{proof}

Consider a representation $(\sigma, s)$ of $(B_I, Y_I)$.
The pair $(\sigma \circ \quot, s \circ \quot)$ forms a representation of $(B, Y)$.
\begin{proposition} \label{p:quotient_representation}
  The mapping $(\sigma, s)\mapsto (\sigma\circ \quot, s\circ \quot)$ defines a bijection between the set of representations of $(B_I, Y_I)$ and the set of representations $(\sigma', s')$ of $(B, Y)$ with kernel of $\sigma'$ containing $I$.
  Moreover, the following statements hold:
  \begin{enumerate}
  \item The representation $(\sigma, s)$ admits a gauge action if and only if $(\sigma\circ \quot, s\circ \quot)$ does.
  \item For all $T \in \cK(Y^{\tup m})$, we have $\psi_{s\circ \quot}(T) = \psi_{s}([T]_I)$.
    Therefore, the equalities $p_{s\circ \quot}^{\tup m} = p_s^{\tup m}$, $Q_{s\circ \quot}^F = Q_s^F$ and $P_{s\circ \quot}^F = P_s^F$ hold.
  \item The representation $(\sigma, s)$ is Nica-covariant if and only if $(\sigma\circ \quot, s\circ \quot)$ is.
  \end{enumerate}
\end{proposition}
\begin{proof}
  Obviously, the kernel of $\pi\circ \quot$ contains $I$.
  Conversely, let $(\sigma', s')$ be a representation of $(B, Y)$ with kernel containing $I$.
  Then, $\sigma'$ and $s'$ descend to maps $\sigma'_I$ and $s'_I$ on $B_I$ and $Y_I$, respectively.
  It is routine check that the resulting pair $(\sigma'_I, s'_I)$ is a representation of $(B_I, Y_I)$.

  The first statement is trivial.
  It is enough to prove the second statement for a rank-one operator $T = \theta_{x,y}$, where $x, y \in Y^{\tup m}$.
  We have $\psi_{\sigma\circ \quot}(T) = s([x]_I)s([y]_I)^* = \psi_{s}(\theta_{[x]_I, [y]_I}) = \psi_{s}([T]_I)$.
  If $(k^{\tup m}_{\lambda})_{\lambda \in \Lambda}$ is a c.a.i. for $Y^{\tup m}$, then $([k^{\tup m}_{\lambda}]_I)_{\lambda \in \Lambda}$ is a c.a.i. for $Y_I^{\tup m}$.
  Hence, $p_{s\circ \quot}^{\tup m} = p_s^{\tup m}$ for all $m\in \N^n$ and the equalities for $Q$ and $P$ follow from the definitions.
  The third statement follows immediately from the second one.
\end{proof}

We use $(\omega_{Y_I}, o_{Y_I})$ to denote the universal CNP-representation of $(B_I, Y_I)$ on $\cNO(Y_I)$.
\begin{lemma} \label{l:invariant_if_cnp}
  A positively invariant ideal $I\subset B$ is invariant if and only if $(\omega_{Y_I}\circ \quot, o_{Y_I}\circ \quot)$ is a CNP-representation of $(B, Y)$.
\end{lemma}
\begin{proof}
  Suppose that $I$ is invariant.
  Then, there exists a gauge-invariant ideal $J\in \I^{\gamma}(\cNO(Y))$ such that $I = J^r$.
  By Proposition \ref{p:quotient_representation}, the quotient representation $(\sigma, s)$ of $(B, Y)$ on $\cNO(Y)/J$ descends to a faithful representation $(\sigma_I, s_I)$ of $(B_I, Y_I)$ on $\cNO(Y)/J$.
  By the co-universal property (Proposition \ref{p:co-universal}) of $\cNO(Y_I)$, this representation defines a canonical epimorphism $\cNO(Y)/J \to \cNO(Y_I)$.
  Together with the quotient map $\cNO(Y) \to \cNO(Y)/J$, this induces a factorization of $(\omega_{Y_I}\circ \quot, o_{Y_I}\circ \quot)$ through $(\omega, o)$. Since $(\omega, o)$ is a CNP-representation, we conclude that  $(\omega_{Y_I}\circ \quot, o_{Y_I}\circ \quot)$ is also a CNP-representation of $(B, Y)$.
  \[
    \begin{tikzcd}
      {(B, Y)} \arrow[rd, "\quot"'] \arrow[rrr, "{(\omega, o)}", Rightarrow] \arrow[rrrd, "{(\sigma, s)}", Rightarrow] \arrow[rrrdd, "{(\omega_{Y_I}\circ \quot, o_{Y_I}\circ \quot)}"', Rightarrow, bend right=49] &                                                                                                               &  & \cNO(Y) \arrow[d, two heads]           \\
      & {(B_I, Y_I)} \arrow[rrd, "{(\omega_{Y_I}, o_{Y_I})}"', Rightarrow] \arrow[rr, "{(\sigma_I, s_I)}"', Rightarrow] &  & \cNO(Y)/J \arrow[d, two heads, dashed] \\
      &                                                                                                               &  & \cNO(Y_I)
    \end{tikzcd}\]

  Conversely, suppose that $(\omega_{Y_I}\circ \quot, o_{Y_I}\circ \quot)$ is a CNP-representation of $(B, Y)$.
  Since it admits a gauge action, the representation induces a gauge-invariant homomorphism $\cNO(Y) \to \cNO(Y_I)$.
  Then, the kernel $J$ of this epimorphism is a gauge-invariant ideal of $B$ such that $I = J^r$.
  Hence, $I$ is invariant.
\end{proof}

We define ideals
\[
  L^i_I = (Y^{\tupi})^{-1}(I) = \{ b\in B \colon bY^{\tupi} \subset Y^{\tupi}I\},
\]
and $L^F_I = \bigcap_{i\in F} L^i_I$.

\begin{definition} \label{d:neg-inv}
  An ideal $I\subset B$ is called \emph{negatively invariant} if $L^F_I\cap \cI^F_Y \subset I$ for all $F\in \cF$.
\end{definition}
Again, in the rank one case, Definition \ref{d:neg-inv} transforms to just $Y^{-1}(I)\cap \cI^{\{1\}}_Y \subset I$, which is the definition of a negatively invariant ideal in \cite[Definition 4.8]{K2007}.

\begin{lemma} \label{l:preimage_of_cnp_ideals}
  Let $\cJ^F_{Y_I}$ and $\cI^F_{Y_I}$ be the (pre)-CNP-ideals defined in \eqref{e:def_of_j} and \eqref{e:def_of_i} corresponding to $(B_I, Y_I)$.
  Then,
  \begin{align*}
    \quot^{-1}(\cJ_{Y_I}^F) & = \{ b\in B \colon [\varphi^{\tupi}_Y(b)]_I\in \cK(Y^{\tupi}_I) \text{ for all } i\in [n]\, \text{ and } bL^F_I \subset I\},              \\
    \shortintertext{and}
    \quot^{-1}(\cI_{Y_I}^F) & = \{ b\in \quot^{-1}(\cJ_{Y_I}^F) \colon bY^{\tup m} \subset Y^{\tup m}(\quot^{-1}(\cJ_{Y_I}^F)) \text{ for all } \tup m\perp \tup 1_F\}.
  \end{align*}
  Consequently, $[\cI_{Y}^F]_I\subset \cI_{Y_I}^F$ for all $F\in \cF$ if and only if $I$ is negatively invariant.
\end{lemma}
\begin{proof}
  Firstly, observe that $[L^i_I]_I = \ker \varphi^{\tupi}_{Y_I}$.
  Indeed, $[b]_I$ is in the kernel if and only if $bY^{\tupi} \subset Y^{\tupi}I$.
  Analogously, $[b]_I \perp \bigcap_{i\in F} \ker \varphi^{\tupi}_{Y_I}$ if and only if $bL^F_I \subset I$.
  The expression for $\quot^{-1}(\cJ_{Y_I}^F)$ then follows from the definition of $\cJ_{Y_I}^F$.

  An element $b\in \quot^{-1}(\cJ_{Y_I}^F)$ is in $\quot^{-1}(\cI_{Y_I}^F)$ if and only if $[b]_I Y^{\tup m}_I \subset Y^{\tup m}_I \cJ_{Y_I}^F$ for all $\tup m\perp \tup1_F$.
  This is equivalent to $bY^{\tup m} \subset Y^{\tup m}(\quot^{-1}(\cJ_{Y_I}^F) + I) = Y^{\tup m}(\quot^{-1}(\cJ_{Y_I}^F))$.
  The last equality follows from the fact that $I = \quot^{-1}(0) \subset \quot^{-1}(\cJ_{Y_I}^F)$.
  This proves the expression for the preimage of $\cI_{Y_I}^F$.

  Now, consider an arbitrary $b\in \cI_Y^F$.
  Then, the condition $[\varphi^{\tupi}_Y(b)]_I\in \cK(Y^{\tupi}_I)$ is always satisfied.
  Observe that since $\cI_Y^F Y^{\tup m} \subset Y^{\tup m}\cI_Y^F$ for all $\tup m \perp \tup1_F$, we have $[\cI_{Y}^F]_I\subset \cI_{Y_I}^F$ if and only if $[\cI_{Y}^F]_I\subset \cJ_{Y_I}^F$.
  Therefore, $[\cI_{Y}^F]_I\subset \cI_{Y_I}^F$ if and only if $bL^F_I \subset I$ for all $b\in \cI_Y^F$ or, equivalently, $\cI^F_Y L^F_I = L^F_I\cap \cI^F_Y\subset I$, which is exactly the definition of negative invariance.
  This proves the second statement.
\end{proof}

\begin{theorem} \label{t:invariant_if_pos_and_neg}
  An ideal $I\subset B$ is invariant if and only if it is positively invariant and negatively invariant.
\end{theorem}
\begin{proof}
  Suppose that $I$ is invariant.
  Then, $I$ is positively invariant by Proposition \ref{p:quotient_representation} and $(\omega_{Y_I}\circ \quot, o_{Y_I}\circ \quot)$ is a CNP-representation of $(B, Y)$ by Lemma \ref{l:invariant_if_cnp}.
  To prove that $I$ is negatively invariant, by Lemma \ref{l:preimage_of_cnp_ideals}, it suffices to show that $[\cI_{Y}^F]_I\subset \cI_{Y_I}^F$ for all $F\in \cF$.

  Let $b\in \cI_{Y}^F$ be arbitrary.
  Since $(\omega_{Y_I}\circ \quot, o_{Y_I}\circ \quot)$ is a CNP-representation, we have $\omega_{Y_i}([b]_I)Q^F_{o_{Y_I}\circ \quot} = 0$.
  By Proposition \ref{p:quotient_representation}, we have $Q^F_{o_{Y_I}\circ \quot} = Q^F_{o_{Y_I}}$.
  Hence, we have $\omega_{Y_i}([b]_I)Q^F_{o_{Y_I}} = 0$.
  Since $b\in \cI_{Y}^F\subset \bigcap_{i = 1}^n (\varphi^{\tupi})^{-1}(\cK(X^{\tupi}))$ by definition, we can apply Lemma \ref{l:cIF_is_max} and get $[b]_I\in \cI_{Y_I}^F$.
  We conclude that $[\cI_{Y}^F]_I\subset \cI_{Y_I}^F$ and thus $I$ is negatively invariant.

  Conversely, suppose that $I$ is positively and negatively invariant.
  Then, by Lemma \ref{l:preimage_of_cnp_ideals}, $[\cI_{Y}^F]_I\subset \cI_{Y_I}^F$ for all $F\in \cF$.
  This implies that any $b\in \cI_{Y}^F$ satisfies the equation
  \[
    \omega_{Y_I}([b]_I)Q^F_{o_{Y_I}} = \omega_{Y_I}([b]_I)Q^F_{o_{Y_I}\circ \quot} = 0.
  \]
  Hence, $(\omega_{Y_I}\circ \quot, o_{Y_I}\circ \quot)$ is a CNP-representation of $(B, Y)$.
  By Lemma \ref{l:invariant_if_cnp}, $I$ is invariant.
\end{proof}

\begin{proposition} \label{p:separating_criterion}
  Let $I\subset B$ be a positively invariant ideal. Suppose that $[\cI^F_Y]_I = \cI^F_{Y_I}$ for all $F\in \cF$. Then, $I$ is separating.
\end{proposition}
\begin{proof}
  Let $J_{\min}$ be the intersection of all ideals $J\in \I^{\gamma}(\cNO(Y))$ such that $I = J^r$.
  Then, $J_{\min}$ is gauge-invariant, $J_{\min}^r = I$ and $J_{\min}$ is the minimal ideal with these properties.
  Let $F\in \cF$ and $a \in \cI_{Y_I}^F$ be an arbitrary element.
  By assumption, there exists $\hat a \in \cI_Y^F$ such that $[\hat a]_I = a$.
  Therefore, we have
  \[
    0 = [\omega_Y(\hat a) Q_{o_Y}^F]_{J_{\min}} = \omega'(a) Q_{o'}^F \in \cNO(Y)/J_{\min},
  \]
  where $(\omega', o')$ is the representation of $(B_I, Y_I)$ on $\cNO(Y)/J_{\min}$ induced by $(\omega_Y, o_Y)$.
  Since $F$ and $a$ were arbitrary, we conclude that $(\omega', o')$ is a CNP-representation of $(B_I, Y_I)$.
  Since it is injective and gauge-invariant, it induces an isomorphism $\cNO(Y)/J_{\min} \cong \cNO(B_I)$.

  Suppose that $J\in \I^{\gamma}(\cNO(Y))$ is some other ideal such that $I = J^r$.
  Then, the representation of $(B_I, Y_I)$ on $\cNO(Y)/J\cong (\cNO(Y)/J_{\min})/(J/J_{\min})$ is also injective and gauge-invariant.
  By the GIUT, we conclude that $J/J_{\min} = 0$ or $J = J_{\min}$.
  Hence, $I$ is separating.
\end{proof}

\begin{corollary}
  Let $(B, Y)$ be a regular (proper and injective) $\N^n$-product system. Then, any invariant ideal $I\subset B$ is separating. Consequently, $I \mapsto I^i$ is a bijection $\Iinv(B) \to \I^{\gamma}(\cNO(Y))$.
\end{corollary}
\begin{proof}
  Since $Y$ is regular, we have $\cI^F_Y = B$ for all $F\in \cF$.
  Let $I\subset B$ be an invariant ideal.
  Then, $I$ is negatively invariant and by Lemma \ref{l:preimage_of_cnp_ideals}, $B_I = [\cI^F_Y]_I \subset \cI^F_{Y_I}$ for all $F\in \cF$.
  Hence, we have $[\cI^F_Y]_I = \cI^F_{Y_I}$ for all $F\in \cF$ and by Proposition \ref{p:separating_criterion}, $I$ is separating.
  The rest of the proof is straightforward.
\end{proof}
\begin{remark}
  In case of regular $\N^n$-product systems, the set of invariant ideals has a particularly nice description: an ideal $I\subset B$ is invariant (and separating) if and only if $I = (Y^{\tupi})^{-1}(I)$ for all $i\in [n]$.
\end{remark}

\section{Gauge-invariant ideals}
\label{s:gauge-invariant-ideals}
Let $(A, X)$ be a proper product system of rank $n$.
This means that the left action of $A$ is given by homomorphisms $\varphi^{\tup m}\colon A\hookrightarrow \cK(X^{\tup m})$.
Such a system is always strongly compactly aligned.
We denote by $(\tau, t)$ the representation of $(A, X)$ on $\cNT(X)$.

\subsection{Invariant families, T-families, and O-families of ideals} \label{ss:families-of-ideals}
\begin{definition} \label{d:inv-fam}
  A collection $K=\{K^F\}_{F\in \cF}$ of ideals of $A$ is an \emph{invariant family} if it satisfies the condition
  \[
    K^G = (X^{\tupi})^{-1}(K^G \cap K^{G\cup \{i\}})
  \]
  for all $G\in \cF$ and $i\in [n]\setminus G$.
  We write $K_1 \preceq K_2$ if $K_1^F \subset K_2^F$ for all $F\in \cF$.
\end{definition}
\begin{definition} \label{d:to-fam}
  A collection $I = \{I^F\}_{F\in \cF}$ of ideals of $A$ is a \emph{T-family} if it satisfies the condition
  \[
    I^F = (X^{\tupi})^{-1}(I^F) \cap I^{F\cup \{i\}}
  \]
  for all $F\in \cF$ and $i\in [n]\setminus F$.
  We write $I_1 \preceq I_2$ if $I_1^F \subset I_2^F$ for all $F\in \cF$.
  A T-family $I$ is called an \emph{O-family} if $\cI \preceq I$.
\end{definition}
It is easy to see that in the case $n=1$, the notions of T-families and O-families coincide with Katsura's T-pairs and O-pairs \cite{K2007}. 

For an invariant family $K$ we define a family $I_K$ by $I_K^F \coloneqq \bigcap_{G\supset F} K^G$.
Analogously, for a T-family $I$ we define a family $K_I$ by
\[
  K_I^F \coloneqq (X^{\tup 1 - \tup 1_F})^{-1}(I^F).
\]
Here and in the following, by $(X^{\tup 0})^{-1}$ we mean the identity map on the ideals of $A$.
\begin{lemma} \label{l:H-GIK}
  Let $F \subset H$ be finite subsets of $[n]$.
  Then, we have
  \begin{enumerate}
  \item \[ \bigcap_{G, H\supset G\supset F} (X^{\tup 1_H - \tup 1_G})^{-1}(I^G) = I^F;\]
  \item \[ (X^{\tup 1_H - \tup 1_F})^{-1}\left(\bigcap_{G, H\supset G\supset F} K^G\right) = K^F. \]
  \end{enumerate}
\end{lemma}
\begin{proof}
  We prove both statements simultaneously by induction on $k = |H| - |F|$.
  If $k = 0$, then $H = G$ and the statements are trivial.

  Suppose that $k>0$ and the statements hold for all $H'$ with $|H'| - |F| < k$.
  Let $i$ be any element in $H\setminus F$ and let $H' = H\setminus \{i\}$.
  Then, we have
  \[
    \bigcap_{G, H\supset G\supset F} (X^{\tup 1_H - \tup 1_G})^{-1}(I^G) = \bigcap_{G, H'\supset G\supset F} (X^{\tup 1_{H'} - \tup 1_G})^{-1}((X^{\tupi})^{-1}(I^{G}) \cap I^{G\cup \{i\}}) =\]\[= \bigcap_{G, H'\supset G\supset F} (X^{\tup 1_{H'} - \tup 1_G})^{-1}(I^G) = I^F.
  \]
  Here we used Lemma \ref{l:inv_composition} in the first equality and the induction hypothesis in the last one.
  The second statement is proved analogously.
\end{proof}
\begin{proposition} \label{p:T_and_inv_families_equiv}
  The maps $K\mapsto I_K$ and $I\mapsto K_I$ are mutually inverse lattice isomorphisms between the set of invariant families and the set of T-families.
\end{proposition}
\begin{proof}
  We first show that $I_K$ is a T-family.
  First, we compute
  \[
    (X^{\tupi})^{-1}\left(I_K^F\right) = (X^{\tupi})^{-1}(\bigcap_{G\supset F} K^G) = \bigcap_{G\supset F, i\notin G} (X^{\tupi})^{-1}(K^G \cap K^{G\cup \{i\}}) = \bigcap_{G\supset F, i\notin G} K^G
  \]
  and
  \[
    (X^{\tupi})^{-1}(I_K^F)\cap I_K^{F\cup \{i\}} = \bigcap_{G\supset F, i\notin G} K^G \cap \bigcap_{G\supset F\cup \{i\}} K^G = \bigcap_{G\supset F} K^G = I_K^F.
  \]

  Now, we show that $K_I$ is an invariant family.
  For all $i\notin G$, we have
  \[
    (X^{\tupi})^{-1}(K_I^G\cap K^{G\cup \{i\}}) = (X^{\tupi})^{-1}((X^{\tup 1 - \tup 1_G})^{-1}(I^G) \cap (X^{\tup 1 - \tup 1_{G\cup \{i\}}})^{-1}(I^{G\cup \{i\}}))= \]\[
    = (X^{\tup 1 - \tup 1_G})^{-1}\left((X^{\tupi})^{-1}(I^G)\cap I^{G\cup \{i\}}\right) = (X^{\tup 1 - \tup 1_G})^{-1}(I^G) = K_I^G.
  \]
  Here, the second equality follows from Lemma \ref{l:inv_composition}.

  Our next goal is to show that $I_{K_I} = I$.
  Indeed, we have
  \[
    I_{K_I}^F = \bigcap_{G\supset F} (X^{\tup 1 - \tup 1_G})^{-1}(I^G),
  \]
  which is equal to $I$ by the special case $H=[n]$ of Lemma \ref{l:H-GIK}.

  Finally, we show that $K_{I_K} = K$.
  Analogously, we have
  \[
    K_{I_K}^F = (X^{\tup 1-\tup 1_F})^{-1}\left(\bigcap_{G\supset F} K^G\right) = K^F
  \]
  by Lemma \ref{l:H-GIK}. This completes the proof.
\end{proof}

\subsection{Extended product system} \label{ss:extended-product-system}
Let $K$ be an invariant family.
Define a \Cs{} $\bA_K \coloneqq \bigoplus_{F\in \cF} A/K^F$ together with a diagonal morphism $\Delta_K\colon A\to \bA_K$ given by $\Delta_K(a) = ([a]_{K^F})_{F\in \cF}$.
Consider the induced Hilbert $\bA_K$-modules \[\bX_K^{\tup m} \coloneqq X^{\tup m}\otimes_A \bA_K \cong \bigoplus_{F\in \cF} X^{\tup m}_{K^F}.\]
The last isomorphism follows from Lemma \ref{l:direct_sum}.
\begin{notation}
  In the following, we will frequently face the situation when we have some vector space $V$ and a family of subspaces $W^F\subset V$ indexed by $F\in \cF$.
  We use bold letters like $\boldsymbol{v}$ for elements of $\boldsymbol{V}_W = \bigoplus_{F \in \cF} V/W^F$.
  For $\boldsymbol{v} \in \boldsymbol{V}_W$, we denote by $\boldsymbol{v}_F \in V/W^F$ the component of $\boldsymbol{v}$ corresponding to $F$.
  For an element $v\in V/W^F$, we denote by $\hat v$ an arbitrary lift of $v$ to $V$.

  Suppose that there are subspaces $\boldsymbol{U}^F \subset V/W^F$ for all $F\in \cF$.
  We denote by $\boldsymbol{U} = \bigoplus_{F\in \cF} \boldsymbol{U}^F \subset \boldsymbol{V}$ their direct sum.
  Conversely, if $\boldsymbol{U} \subset \boldsymbol{V}$ is a subspace of the above form, then we denote by $\boldsymbol{U}^F$ its $F$-component.
\end{notation}
As an example of the above notation, consider the case $V=X^{\tup m}$ and $W^F = X^{\tup m}K^F$.
Then, we use $\bx, \by$ for elements of $\bX^{\tup m}_K = \bigoplus_{F\in \cF} X^{\tup m}_{K^F}$ and $\hat\bx_F$ stands for an element of $X^{\tup m}$ such that $[\hat\bx_F]_{K^F}$ is the $F$-component of $\bX$.

We would like to define a left action of $\bA_K$ on $\bX_K$ to obtain a product system $(\bA_K, \bX_K)$.
By Lemma \ref{l:direct_sum}, we have $\cK(\bX^{\tup m}_K) = \bigoplus_{F\in \cF} \cK(X^{\tup m}_{K^F}) = \bigoplus_{F\in \cF} \cK(X^{\tup m})/\cK(X^{\tup m}K^F)$.
\begin{lemma} \label{l:act_of_K}
  Let $F \in \cF$ and $\tup m \in \N^n$.
  We have $\varphi^{\tup m}(K^{F\setminus \supp \tup m}) \subset \cK(X^{\tup m}K^F)$ and, hence, the expression \[b\mapsto [\varphi^{\tup m}(\hat b)]_{K^F}\in \cK(X^{\tup m}_{K^F})\] is well-defined for any $b = [\hat b]_{K^{F\setminus \supp\tup m}}\in A/K^{F\setminus \supp \tup m}$.
  It defines a homomorphism $A/K^{F\setminus \supp \tup m} \to \cK(X^{\tup m}_{K^F})$.
\end{lemma}
\begin{proof}
  The inclusion $\varphi^{\tup m}(K^{F\setminus \supp \tup m}) \subset \cK(X^{\tup m}K^F)$ is equivalent to $K^{F\setminus \supp \tup m}X^{\tup m}\subset X^{\tup m}K^F$ and to $(X^{\tup m})^{-1}(K^F)\supset K^{F\setminus \supp \tup m}$.
  We will show this by induction on $|\tup m|$.
  Suppose that $\tup m = \tup 1_i$ for some $i \in [n]$.
  Let us prove that $(X^{\tupi})^{-1}(K^F) \supset K^{F\setminus \{i\}}$.
  Indeed, if $i\in F$, then $(X^{\tupi})^{-1}(K^F) \supset (X^{\tupi})^{-1}(K^F\cap K^{F\setminus \{i\}}) = K^{F\setminus \{i\}}$ by the definition of invariant family.
  Analogously, if $i\notin F$, then $(X^{\tupi})^{-1}(K^F) \supset (X^{\tupi})^{-1}(K^F\cap K^{F\cup \{i\}}) = K^{F} = K^{F\setminus \{i\}}$.
  This shows the base of induction.

  For induction step, write $X^{\tup m} = X^{\tup m - \tup 1_i}\otimes_A X^{\tup 1_i}$ for some $i\in \supp \tup m$.
  Then, we have
  \[
    (X^{\tup m})^{-1}(K^F) = (X^{\tup m - \tup 1_i})^{-1}((X^{\tup 1_i})^{-1}(K^F)) \supset (X^{\tup m - \tup 1_i})^{-1}(K^{F\setminus \{i\}}) \supset\]\[\supset K^{(F \setminus \{i\})\setminus \supp (\tup m - \tup 1_i)} = K^{F\setminus \supp \tup m}.
  \]
  Here we used Lemma \ref{l:inv_composition} and the induction hypothesis twice.
  We have proved the first claim.

  Now, the proved inclusion implies that the kernel of the homomorphism $[-]_{K^F}\circ \varphi^{\tup m} \colon A \to \cK(X^{\tup m})$ contains $K^{F\setminus \supp \tup m}$.
  Therefore, it descends to a homomorphism $A/K^{F\setminus \supp \tup m} \to \cK(X^{\tup m}_{K^F})$ given by the formula in the statement.
\end{proof}

\begin{proposition} \label{p:extended_prod_sys}
  For any invariant family $K$ there is a proper product system $(\bA_K, \bX_K)$ with a left action of $\bA_K$ given componentwise by
  \[
    \varphi^{\tup m}_K(\ba)_F = [\varphi^{\tup m}(\hat\ba_{F\setminus \supp \tup m})]_{K^F}\text{ for all } \ba \in \bA_K,\text{ and } F\in \cF.
  \]
  That is, the left action on the $F$-component is given by the action of the $F\setminus \supp \tup m$-component.
\end{proposition}
\begin{proof}
  We know by Lemma \ref{l:act_of_K} that the formula in the statement gives a well-defined homomorphism $\varphi^{\tup m}_K\colon \bA_K \to \cK(X^{\tup m}_K)$.
  Hence, it defines a left action of $\bA_K$ on $\bX_K^{\tup m}$.

  Let us analyze the tensor product $X^{\tup k}_K\otimes_{\bA_K} X^{\tup m}_K$ for $\tup k, \tup m \in \N^n$.
  By Lemma \ref{l:tensor_over_sum}, we have
  \[
    (\bX^{\tup k}_{K} \otimes_{\bA_K} \bX^{\tup m}_K)^F = \bigoplus_{G\in \cF} X^{\tup k}_{K^G}\otimes_{A_{K^G}} X^{\tup m}_{K^F} = X^{\tup k}_{K^{F\setminus \supp \tup m}}\otimes_{A_{K^{F\setminus \supp \tup m}}} X^{\tup m}_{K^F}.
  \]
  The last equality follow from the fact that $A_{K^G}$ acts by zero on $X^{\tup m}_{K^F}$ unless $G = F\setminus\supp\tup m$.
  Now, define a multiplication map $(\bX^{\tup k}_K \otimes_{\bA_K} \bX^{\tup m}_K)^F = X^{\tup k}_{K^{F\setminus \supp \tup m}}\otimes_{A_{K^{F\setminus \supp \tup m}}} X^{\tup m}_{K^F} \to (\bX^{\tup k + \tup m}_K)^F = X^{\tup k+\tup m}_{K^F}$ by the formula
  \[
    x\cdot y = [\hat x\cdot \hat y]_{K^F}
  \]
  for $x\in X^{\tup k}_{K^{F\setminus \supp \tup m}}$ and $y\in X^{\tup m}_{K^F}$.
  It is well-defined by Lemma \ref{l:act_of_K}.
  Moreover, it is unitary since it is induced by the unitary multiplication $X^{\tup m} \otimes_A X^{\tup k} \to X^{\tup m + \tup k}$.
  This defines multiplication isomorphism of Hilbert $\bA_K$-modules $\bX^{\tup k}_K \otimes_{\bA_K} \bX^{\tup m}_K \cong \bX^{\tup k + \tup m}_K$.
  It is also easy to see that the isomorphism is compatible with the left actions of $\bA_K$.
  This shows that $(\bA_K, \bX_K)$ is a product system, which is proper by definition.
\end{proof}
\begin{remark}
  The construction of the product system $(\bA_K, \bX_K)$ above is inspired by Katsura's \cite[Definition 6.1]{K2007}.
  However, it is different from Katsura's extended product system even for rank 1 product systems and his construction does not have a straightforward generalization to the higher rank case.
  Katsura works solely with T-pairs of ideals, while the equivalent picture of invariant families turns out to be more useful for the higher rank case.
  The extended product system is our main tool for the definition of relative CNP-algebras and classification of gauge-invariant ideals.
\end{remark}

Our next goal is to compute the CNP-ideals $\cI_K$ of $(\bA_K, \bX_K)$.
We have
\[
  \ker \varphi^{\tupi}_K = \{ \ba\in \bA_K\colon [\varphi^{\tupi}(\hat\ba_{F\setminus \{i\}})]_{K^F} = 0 \text{ for all } F\in \cF\}.
\]
Therefore, for any $\ba\in \ker \varphi^{\tupi}_K$, there is no condition on $\ba_G$ with $i \in G$ and there are two conditions for $a_G$ with $i\notin G$.
These conditions are $[\varphi^{\tupi}(\ba_{G})]_{K^G} = 0$ and $[\varphi^{\tupi}(\ba_{G})]_{K^{G\cup \{i\}}} = 0$.
They can be rewritten as $\hat{\ba}_GX^{\tupi}\subset X^{\tupi}(K^G\cap K^{G\cup \{i\}})$.
In its turn, this is equivalent to $\hat{\ba}_G\in (X^{\tupi})^{-1}(K^G\cap K^{G\cup \{i\}}) = K^G$.
Therefore, we have $\hat{\ba}_G \in K^G$ and hence $\ba_G = 0$.
Finally, we obtain
\[
  \ker \varphi^{\tupi}_K = \{ \ba\in \bA_K\colon \ba_G = 0 \text{ for all } G\in \cF \text{ with } i\notin G\}
\]
and
\[
  \bigcap_{i\in F} \ker \varphi^{\tupi}_K = \{ \ba\in \bA_K\colon \ba_G = 0 \text{ for all } G\in \cF \text{ with } F\not\subset G \}
\]
for any $F\in \cF$.
\begin{lemma} \label{l:cnp_ideals_of_AK}
  For any $F\in \cF$, we have
  \[
    \cI_K^F = \cJ_K^F = \{ \ba\in \bA_K\colon \ba_G = 0 \text{ for all } G\in \cF \text{ with } F\subset G\}.
  \]
\end{lemma}
\begin{proof}
  Since the product system is proper, we have
  \[
    \cJ_K^F = \left(\bigcap_{i\in F} \ker \varphi^{\tupi}_K\right)^{\perp} = \{ \ba\in \bA_K\colon \ba_G = 0 \text{ for all } G\in \cF \text{ with } F\subset G \}.
  \]
  To prove the equality $\cI_K^F = \cJ_K^F$, it is enough to show that $\cJ_K^F \bX^{\tupi}_K \subset \bX^{\tupi}_K \cJ_K^F$ for any $i\notin F$.
  Indeed, in this case we have $(X^{\tupi})^{-1}(\cJ^F_K) \supset \cJ^F_K$ for all $i\notin F$ and hence $(X^{\tup m})^{-1}(\cJ^F_K) \supset \cJ^F_K$ for all $\tup m \perp \tup 1_F$ by Lemma \ref{l:inv_composition}.
  Then, the formula \eqref{e:def_of_i} implies $\cI_K^F = \cJ_K^F$.

  Let $\ba\in \cJ_K^F$ and $\bx\in X^{\tupi}_K$ be arbitrary.
  Then, we have
  \[
    (\ba\cdot \bx)_G = \ba_{G\setminus \{i\}}\cdot \bx_G \text{ for all } G\in \cF.
  \]
  This equals zero if $F\subset G\setminus \{i\}$, which is equivalent to $F\subset G$ since $i\notin F$.
  Therefore, we have
  \[
    \cJ_K^F \bX^{\tupi}_K \subset \{\bx\in \bX^{\tupi}_K\colon \bx_G = 0 \text{ for all } G\in \cF \text{ with } F\subset G\} = \bX^{\tupi}_K \cJ_K^F.
  \]
  We conclude that $\cI_K^F = \cJ_K^F$.
\end{proof}

We now classify $\bX_K$-invariant ideals in $\bA_K$.
Since $\bA_K$ is a direct sum of $C^*$-algebras, every ideal in $\bA_K$ is a direct sum of ideals in $(\bA_K)^F = A/K^F$.
Moreover, ideals in $A/K^F$ are in bijective correspondence with ideals in $A$ containing $K^F$.
For a family of ideals $\{N^F\}_{F\in \cF}$ of $A$ with $N^F\supset K^F$, we denote by $N/K$ the ideal $\bigoplus_{F\in \cF} N^F/K^F$ of $\bA_K$.

\begin{proposition} \label{p:invariant_ideals_of_AK}
  An ideal $N/K$ of $\bA_K$ is invariant if and only if $N$ is an invariant family.
  There is a canonical isomorphism of product systems $((\bA_K)_{N/K}, (\bX_K)_{N/K})\cong (\bA_N, \bX_N)$ and we identify those.
  Therefore, there is a lattice isomorphism between invariant ideals in $\bA_K$ and invariant families $N\succeq K$.
  Moreover, every invariant ideal of $\bA_K$ is separating.
\end{proposition}
\begin{proof}
  Suppose $N/K$ is an ideal of $\bA_K$. We will find necessary and sufficient conditions for $N/K$ to be positively invariant and negatively invariant.
  Then, we will apply Theorem \ref{t:invariant_if_pos_and_neg} to deduce the conditions for $N/K$ to be invariant and compare them to the definition of invariant families.

  \textbf{Positive invariance.}
  Recall that $N/K$ is positively invariant if and only if $(N/K) \bX^{\tupi}_K \subset \bX^{\tupi}_K (N/K)$ for all $i\in [n]$.
  Componentwise, we have
  \[
    ((N/K) \bX^{\tupi}_K)^G = (N^{G\setminus\{i\}}/K^{G\setminus\{i\}}) X^{\tupi}/X^{\tupi}K^G = (N^{G\setminus\{i\}} X^{\tupi})/(X^{\tupi}K^G \cap N^{G\setminus\{i\}} X^{\tupi})
  \]
  and
  \[
    (\bX^{\tupi}_K (N/K))^G = (X^{\tupi}/X^{\tupi}K^G) (N^G/K^G) = (X^{\tupi} N^G/X^{\tupi}K^G)
  \]
  for all $G\in \cF$.
  The former contains the latter if and only if $N^{G\setminus\{i\}} X^{\tupi} \subset X^{\tupi} N^G$. Combining these conditions for $G = F$ and $G = F\cup \{i\}$, we conclude that $N$ is positively invariant if and only if
  \begin{equation} \label{e:pos_of_NK}
    N^F \subset (X^{\tupi})^{-1}(N^F \cap N^{F\cup \{i\}}) \text{ for all } F\in \cF \text{ and } i\in [n]\setminus F.
  \end{equation}
  In particular, we see that $N/K$ is positively invariant if $N$ is an invariant family.

  \textbf{Negative invariance.}
  We compute the ideals $L^F_{N/K}$.
  For $i\in [n]$, we have
  \[
    L^i_{N/K} = \{ \ba\in \bA_K\colon \ba_{F\setminus \{i\}} X^{\tupi}/X^{\tupi}K^F \subset X^{\tupi} N^F/X^{\tupi}K^F \text{ for all } F\in \cF\}.
  \]
  The inclusion $\ba_{F\setminus \{i\}} X^{\tupi}/X^{\tupi}K^F \subset X^{\tupi} N^F/X^{\tupi}K^F$ is equivalent to the condition $\hat{\ba}_{F\setminus \{i\}} \in (X^{\tupi})^{-1}(N^F)$, where $\hat{\ba}_{F\setminus \{i\}}$ is an arbitrary lift of $\ba_{F\setminus \{i\}}$ to $A$.
  Therefore, we have
  \begin{align*}
    L^i_{N/K}               & = \{ \ba\in \bA_K\colon \hat{\ba}_{G} \in (X^{\tupi})^{-1}(N^G \cap N^{G\cup \{i\}}) \text{ for all } G\in \cF, i\notin G\},         \\
    L^F_{N/K}               & = \{ \ba\in \bA_K\colon \hat{\ba}_{G} \in (X^{\tupi})^{-1}(N^G \cap N^{G\cup \{i\}}) \text{ for all } G\in \cF, i\in F\setminus G\}, \\
    \shortintertext{and}
    L^F_{N/K}\cap \cI_{K}^F & = \{ \ba\in L^F_{N/K}\colon \hat{\ba}_{G} = 0 \text{ for all } G\in \cF \text{ with } F\subset G\}.
  \end{align*}
  The condition for negative invariance $L^F_{N/K}\cap \cI_{K}^F \subset N/K$ is therefore equivalent to
  \[
    (X^{\tupi})^{-1}(N^G \cap N^{G\cup \{i\}}) \subset N^G \text{ for all } G\in \cF, i\in F\setminus G.
  \]
  Since this condition should hold for all $F\in \cF$, we conclude that $N/K$ is negatively invariant if and only if
  \begin{equation} \label{e:neg_of_NK}
    (X^{\tupi})^{-1}(N^G \cap N^{G\cup \{i\}}) \subset N^G \text{ for all } G\in \cF, i\in [n]\setminus G.
  \end{equation}

  \textbf{Invariance.}
  Combining \eqref{e:pos_of_NK} and \eqref{e:neg_of_NK}, we see that $N/K$ is invariant if and only if
  \[
    N^F = (X^{\tupi})^{-1}(N^F \cap N^{F\cup \{i\}}) \text{ for all } F\in \cF \text{ and } i\in [n]\setminus F.
  \]
  This is exactly the condition for $N$ to be an invariant family.
  Moreover, this shows that there is a lattice isomorphism between the lattice of invariant families containing $K$ and the lattice of invariant ideals in $(\bA_K, \bX_K)$.

  \textbf{Separation.}
  Let us use Proposition \ref{p:separating_criterion} to check that $N/K$ is separating for any invariant family $N\succeq K$.
  First, observe that $(\bA_K/(N/K), \bX_K/(N/K))$ is isomorphic to $(\bA_N, \bX_N)$.
  By Lemma \ref{l:cnp_ideals_of_AK}, we know exactly how the CNP-ideals of $\bA_K$ and $\bA_N$ look like.
  The equality $[\cI^F_K]_{N/K} = \cI^F_{N/K}$ is trivial from the description of these ideals.
\end{proof}
Proposition \ref{p:invariant_ideals_of_AK} classifies gauge-invariant ideals in $\cNO(\bX_K)$.
We will soon see that $\cNO(\bX_K)$ is isomorphic to a certain gauge-invariant quotient of $\cNT(X)$.
This is how we obtain a classification of gauge-invariant ideals in $\cNT(X)$ and $\cNO(X)$.

\subsection{Relative Cuntz-Nica-Pimsner algebra}
We assume henceforth that our product systems are proper. By Lemma \ref{l:prop_of_projs}.\ref{l:prop_of_projs:p_of_a}, for any representation $(\sigma, s)$ of a proper product system $(B, Y)$ in a \Cs{} $D$, the elements $\sigma(a)\cdot p^{\tup m}_s, \sigma(a)\cdot Q^F_s$, and $\sigma(a)\cdot P^F_s$ are in $D$ for any $a\in B$, nonzero $\tup m \in \N^n$, and $F\in \cF$.
We will use this fact without further elaboration.
\begin{definition} \label{d:relative-algebra}
  Let $I$ be a T-family of ideals of $A$. We define a gauge-invariant ideal $\cC_I\subset \cNT(X)$ to be the one generated by the elements
  \[
    \tau(a)\cdot Q_t^F = \sum_{\tup 0 \leq \tup m \leq \tup 1_F} (-1)^{|\tup m|} \psi_t^{\tup m}(\varphi^{\tup m}(a)) 
  \]
  for all $F\in \cF$ and $a\in I^F$.
  We call the algebra $\cNO(X, I)\coloneqq \cNT(X)/\cC_I$ the \emph{$I$-relative Cuntz-Nica-Pimsner algebra of $(A, X)$}.

  A Nica-covariant representation $(\sigma, s)$ is called \emph{$I$-Cuntz-Nica-Pimsner-covariant} if $\tau(a)\cdot Q^F_s = 0$ for all $F\in \cF$ and $a\in I^F$.
  Equivalently, it is $I$-relative if and only if the induced representation $\sigma\times_0 s$ of $\cNT(X)$ factors through $\cNO(X, I)$.
  We denote the corresponding representation of $\cNO(X, I)$ by $\sigma \times_I s$.
\end{definition}
From the definition, it is obvious that $\cNT(X) \cong \cNO(X, 0)$ and $\cNO(X) \cong \cNO(X, \cI)$, where $\cI$ is the CNP-familiy of ideals from Definition \ref{d:cnp-rep}.
Also, observe that $\cC_I \subset \cC_{I'}$ if $I \preceq I'$.
\begin{lemma} \label{l:CI_gen_by_KI}
  The ideal $\cC_I$ is also the ideal generated by the elements $\tau(a)\cdot P^F_t$ for all $F\in \cF$ and $a\in K_I^F$.
\end{lemma}
\begin{proof}
  Denote by $\cC_I'$ the ideal generated by the subspaces $\tau(K^F_I)\cdot P^F_t$.
  We first show that $\cC_I\subset \cC_I'$.
  Proposition \ref{p:T_and_inv_families_equiv} implies
  \[
    I^F = I_{K_I}^F = \bigcap_{G\supset F} K^G.
  \]
  Therefore, for any $a\in I^F$ and $G\supset F$, we have $a\in K^G$ and the element $\tau(a)\cdot P^G_t$ is in $\cC_I'$.
  Thus,
  \[
    a\cdot Q_t^F = a\cdot \sum_{G\supset F} P_t^G =\sum_{G\supset F} a\cdot P_t^G
  \]
  is also an element of $\cC_I'$.
  We conclude that $\tau(I^F)\cdot Q^F_t \subset \cC_I'$ for all $F\in \cF$ and, hence, the ideal $\cC_I$ lies inside $\cC_I'$.

  To prove the other inclusion, consider arbitrary elements $a\in I^F$, $x,y\in X^{\tup 1 - \tup 1_F}$.
  We have
  \[
    t(x)\tau(a)Q_t^Ft(y)^* = t(x)\tau(a)t(y)^* Q_t^F = \psi_\tau^{\tup 1 - \tup 1_F}(\theta_{xa, y}) Q_t^F \in \cC_I.
  \]
  We have used that $t(y)$ commutes with $Q_t^F$ by Lemma \ref{l:prop_of_projs}.
  Therefore, we have an inclusion $\psi_\tau^{\tup 1 - \tup 1_F}(\cK(X^{\tup 1 - \tup 1_F}I^F))\cdot Q^F_t \subset \cC_I$.
  By definition, we have $K^FX^{\tup 1 - \tup 1_F} \subset X^{\tup 1 - \tup 1_F}I^F$, so $\varphi^{\tup 1 - \tup 1_F}(K^F) \subset \cK(X^{\tup 1 - \tup 1_F}I^F)$.
  It follows that
  \[
    \tau(K^F)P^F_t = \tau(K^F)\prod_{i\notin F}p_\tau^{\tupi}Q^F_t = \psi_\tau^{\tup 1 - \tup 1_F}(\varphi^{\tup 1 - \tup 1_F}(K^F))\prod_{i\in F} Q^F_t \subset \cC_I
  \]
  and $\cC_I'\subset \cC_I$.
  This shows that $\cC_I = \cC_I'$ and the lemma is proved.
\end{proof}
Lemma \ref{l:CI_gen_by_KI} is useful, since covariance conditions coming from invariant families are easier to work with.
We denote the representation of $(A, X)$ on $\cNO(X, I_K)$ by $(\rho_K, r_K)$.
We extend this representation to a representation $(\bar\rho_K, \bar r_K)$ of $(\bA_K, \bX_K)$ on $\cNO(X, I_K)$ by
\begin{equation} \label{e:bar_rho_K}
  \begin{split}
    \bar\rho_K(\ba) &= \sum_{F\in \cF} \rho_K(\hat \ba_F)\cdot P^F_{r_K}, \\
    \bar r_K^{\tup m}(\bx) &= \sum_{F\in \cF} r_K^{\tup m}(\hat \bx_F)\cdot P^F_{r_K}.
  \end{split}
\end{equation}

Recall that we have defined in Section \ref{ss:extended-product-system} the diagonal homomorphism $\Delta_K\colon A\to \bA_K$ as $\Delta_K(a) = ([a]_{K^F})_{F\in \cF}$.
We use the same notation for the diagonal map $\Delta_K\colon X\to \bX_K$ given by $\Delta_K(x) = ([x]_{K^F})_{F\in \cF}$.
\begin{lemma} \label{l:bar_rho_K}
  Formula \eqref{e:bar_rho_K} defines a gauge-equivariant CNP-representation of $(\bA_K, \bX_K)$ on $\cNO(X, I_K)$.
  Moreover, we have $(\rho_K, r_K) = (\bar\rho_K\circ \Delta_K, \bar r_K\circ \Delta_K)$.
\end{lemma}
\begin{proof}
  First, we show that the formula \eqref{e:bar_rho_K} does not depend on the choice of the lift $\hat \ba_F$.
  For this, it is enough to show that $\rho_K(b)\cdot P^F_{r_K} = 0$ for any $b\in K^F$ and $r_K^{\tup m}(x)\cdot P^F_{r_K} = 0$ for any $x\in X^{\tup m}K^F$.
  The first one is trivial, since $\rho_K(b)\cdot P^F_{r_K} = [\tau(b)\cdot P^F_{\tau}]_{\cC_{I_K}}$ and the latter is zero by Lemma \ref{l:CI_gen_by_KI}.
  For the second one, write $x = y\cdot b$ for some $y\in X^{\tup m}$ and $b\in K^F$.
  Then, we have $r_K^{\tup m}(x)\cdot P^F_{r_K} = r_K^{\tup m}(y)\cdot \rho_K(b)\cdot P^F_{r_K} = 0$.

  The map $\bar\rho_K$ is a $*$-homomorphism, since the projections $P_{r_K}^F$ are pairwise orthogonal and $\rho_K$ is a $*$-homomorphism.
  Let us check whether $(\bar\rho_K, \bar r_K)$ agrees with the scalar product and the right action of $\bA_K$.
  Let $\ba\in \bA_K$, $\bx,\by\in \bX_K^{\tup m}$ be arbitrary.
  We have
  \begin{multline*}
    \bar r_K(\bx)\bar \rho_K(\ba) = \sum_{F,G\in \cF} r_K(\hat\bx_F)\cdot P^F_{r_K} \cdot \rho_K(\hat\ba_G)\cdot P^G_{r_K} = \sum_{F,G\in \cF} r_K(\hat\bx_F)\rho_K(\hat\ba_G)\cdot P^F_{r_K}  P^G_{r_K} \\
    = \sum_{F\in \cF} r_K(\bar x_F \cdot \bar a_F)\cdot P^F_{r_K} = \bar r_K(\bx\cdot \ba),
  \end{multline*}
  and
  \begin{multline*}
    \bar r_K(\bx)^*\bar r_K(\by) = \sum_{F,G\in \cF} P^F_{r_K} \cdot r_K(\hat\bx_F)^*\cdot r_K(\hat\by_G)\cdot P^G_{r_K} = \sum_{F,G\in \cF} P^F_{r_K} \cdot \rho_K(\langle \hat\bx_F, \hat\by_G\rangle)\cdot P^G_{r_K} \\
    = \sum_{F, G \in \cF} \rho_K(\langle \bar x_F, \bar y_G\rangle)\cdot P^F_{r_K} \cdot P^G_{r_K} = \sum_{F, G \in \cF} \rho_K(\langle \bar x_F, \bar y_F\rangle)\cdot P^F_{r_K} =\bar\rho_K(\langle \bx, \by\rangle).
  \end{multline*}
  In both cases, we have used the pairwise orthogonality of the $P$-projections and Lemma \ref{l:prop_of_projs}.\ref{l:prop_of_projs:p_of_a}.

  The left action of $\bA$ is more involved. For arbitrary $\ba\in \bA_K$, $\bx\in \bX_K^{\tup m}$, we have
  \[
    \bar\rho_K(\ba)\bar r_K(\bx) = \sum_{F,G\in \cF} \rho_K(\hat\ba_F)\cdot P^F_{r_K} \cdot r_K(\hat\bx_F)\cdot P^G_{r_K}.
  \]
  By Lemma \ref{l:prop_of_projs}.\ref{l:prop_of_projs:PFsPG}, we have $P^F_{r_K} \cdot r_K(\hat\bx_G) P^G_{r_K} = r_K(\hat\bx_G)P^G_{r_K}$ if $F = G\setminus \supp(\tup m)$ and $0$ otherwise.
  Therefore, $\bar r_K$ is a left $\bA_K$-module homomorphism.
  A similar calculation shows that $\bar r_K$ preserves multiplication.
  We conclude that $(\bar\rho_K, \bar r_K)$ is a representation.

  To prove that it is Nica-covariant and CNP we need to determine the projections $p^{\tup m}_{\bar r_K}$.
  We claim that $p^{\tup m}_{\bar r_K} = p^{\tup m}_{r_K}$.
  Indeed, observe that
  \[
    p^{\tup m}_{r_K}\cdot \bar r_K^{\tup m}(\bx) = \sum_{F\in \cF} p^{\tup m}_{r_K} r_K^{\tup m}(\hat \bx_F)\cdot P^F_{r_K} = \bar r_K^{\tup m}(\bx) \text{ for any } \bx\in \bX^{\tup m}.
  \]
  By Lemma \ref{l:prop_of_projs}.\ref{l:prop_of_projs:smallest}, $p_{\bar r_K}^{\tup m}$ is the minimal projection fixing $\rho_K(X^{\tup m})$ with left multiplication, so $p^{\tup m}_{\bar r_K}\leq p^{\tup m}_{r_K}$ holds.
  On the other hand, we have $r_K(X^{\tup m})\subset \bar r_K(X^{\tup m}_K)$, so $p^{\tup m}_{\bar r_K}$ fixes $r_K(X^{\tup m})$ with left multiplication and, hence, $p^{\tup m}_{r_K}\leq p^{\tup m}_{\bar r_K}$.
  We conclude that $p^{\tup m}_{\bar r_K} = p^{\tup m}_{r_K}$ and $(\bar\rho_K, \bar r_K)$ is Nica-covariant since $(\rho_K, r_K)$ is Nica-covariant.

  To prove that $(\bar\rho_K, \bar r_K)$ is CNP, we need to show that $\bar\rho_K(a)\cdot Q^F_{\bar r_K} = \bar \rho_K(a)\cdot Q^F_{r_K} = 0$ for all $a\in \cI_K^F$ and $F\in \cF$.
  By Lemma \ref{l:cnp_ideals_of_AK}, an element $a\in \bA_K$ lies in $\cI_K^F$ if and only if $a_G = 0$ for all $G\supset F$.
  Therefore, we have
  \[
    \bar\rho_K(a)\cdot Q^F_{r_K} = \sum_{G\not\supset F} \rho_K(\hat a_G)\cdot P^G_{r_K}\cdot Q^F_{r_K} = 0.
  \]
  The second equality follows from the fact that $P^G_{r_K}Q^F_{r_K}=0$ for $G\not\supset F$.
  Indeed, we have the factor $p^{\tupi}_{r_K}$ in the definition of $P^G_{r_K}$ and an orthogonal factor $(1 - p^{\tupi}_{r_K})$ in $Q^F_{r_K}$ for $i\in F\setminus G$.
  Hence, the representation is CNP. The last statement is trivial.
\end{proof}

We define maps $\gamma_K\coloneqq \omega_{\bX_K}\circ \Delta_K\colon A\to \cNO(\bX_K)$ and $g_K^{\tup m}\coloneqq \omega_{\bX_K}\circ \Delta_K\colon X^{\tup m}\to \cNO(\bX_K)$.
\begin{lemma}
  The pair $(\gamma_K, g_K)$ is an $I_K$-relative Nica-covariant representation of $(A, X)$ on $\cNO(\bX_K)$.
\end{lemma}
\begin{proof}
  It is obvious that the pair $(\gamma_K, g_K)$ forms a representation.
  Since $g_K(X^{\tup m})\cdot \omega_{\bX_K}(\bA_K) = o_{\bX_K}(\bX_K^{\tup m})$, a projection $p$ satisfies $p\cdot g_K^{\tup m}(x) = g_K^{\tup m}(x)$ for all $x\in X^{\tup m}$ if and only if $p\cdot o_{\bX_K}^{\tup m}(\bx)$ for all $\bx\in \bX^{\tup m}$.
  We deduce from Lemma \ref{l:prop_of_projs}.\ref{l:prop_of_projs:smallest} that $p^{\tup m}_{g_K} = p^{\tup m}_{o_{\bX_K}}$ for all $\tup m\in \N^k$.
  In particular, the representation is Nica-covariant.

  To show that it is $I_K$-relative, it is enough to show that $\Delta_K(I_K^F) \subset \cI_K^F$ for all $F\in \cF$.
  Indeed, in this case we have $\gamma_K(a)\cdot Q^F_{g_K} = \omega_{\bX_K}(\Delta_K(a))\cdot Q^F_{o_{\bX_K}} = 0$ for all $a\in I_K^F$, since $(\omega_{\bX_K}, o_{\bX_K})$ is CNP.

  To prove the inclusion, recall that $I_K^F = \bigcap_{G\supset F} K^G$.
  Therefore, for any element $a\in I_K^F$, we have $a\in K^G$ and $(\Delta_K(a))_G = [a]_{K^G} =0$ for all $G\supset F$.
  But this is exactly the condition that $\Delta_K(a)\in \cI_K^F$ by Lemma \ref{l:cnp_ideals_of_AK}.
  This proves that $(\gamma_K, g_K)$ is $I_K$-relative.
\end{proof}

\begin{proposition} \label{p:iso_of_relative}
  The induced map $\bar \rho_K \times \bar r_K\colon \cNO(\bX_K) \to \cNO(X, I_K)$ is an isomorphism with inverse map $\gamma_K \times_{I_K} g_K\colon \cNO(X, I_K) \to \cNO(\bX_K)$.
  It fits into the commutative diagram
  \[
    \begin{tikzcd}
      \cNO(\bX_K) \arrow[r, "{\bar \rho_K \times \bar r_K}"] \arrow[d, two heads] & {\cNO(X, I_K)} \arrow[d, two heads] \\
      \cNO(\bX_{K'}) \arrow[r, "{\bar \rho_{K'} \times \bar r_{K'}}"]                   & {\cNO(X, I_{K'})}
    \end{tikzcd}
  \]
  for any invariant family $K'\succeq K$.
\end{proposition}
\begin{proof}
  The map $\bar \rho_K \times \bar r_K$ is surjective and gauge-equivariant.
  By the GIUT (Proposition \ref{p:giut}), it is an isomorphism if and only if $\bar\rho_K$ is injective.

  By Proposition \ref{p:invariant_ideals_of_AK}, the kernel of $\bar\rho_K$ is described by an invariant family $N \succeq K$, i.e., $\ker \bar\rho_K = N/K = \bigoplus_{F\in \cF} N^F/K^F\subset \bA_K$.
  Suppose that the kernel is nontrivial, so that $K$ is strictly contained in $N$.
  By the lattice isomorphism between invariant families and T-families, we also have that $I_K$ is strictly contained in $I_N$.
  Hence, there is $G\in \cF$ such that $I_K^G \subsetneq I_N^G$.

  Consider an arbitrary element $a \in I_N^G \setminus I_K^G$.
  Recall that $I_N^G = \bigcap_{H\supset G} N^G$ and $I_K^G = \bigcap_{H\supset G} K^G$.
  Therefore, $a\in N^F$ for all $F\supset G$ and there is at least one $H\supset G$ such that $a\notin K^H$.
  Define an element $\bb \in \bA_K$ by
  \[
    \bb_F =
    \begin{cases}
      [a]_{K^F} & \text{if } F\supset G, \\
      0         & \text{otherwise}.
    \end{cases}
  \]
  By the discussion above, $\bb\subset N/K$ and $\bb_H \neq 0$ for some $H\supset G$.

  We calculate
  \[
    0 = \bar\rho_K(\bb) = \sum_{H\supset G} \rho_K(a) \cdot P^H_{r_K} = \rho_K(a) \cdot Q^G_{r_K}.
  \]
  Furthermore, we have
  \[
    0 = (\gamma_K\times_{I_K} g_K)(\rho_K(a)\cdot Q^G_{r_K}) =  \gamma_K(a)\cdot Q^G_{g_K} = \omega_{\bX_K}(\Delta_K(a))\cdot Q^G_{o_{\bX_K}},
  \]
  which implies that $\Delta_K(a)\in \cI_K^G$ by Lemma \ref{l:cIF_is_max}.
  This means that $\Delta_K(a)_H = [a]_{K^H} = 0$ for all $H\supset G$.
  This is a contradiction, since we have showed above that $a\notin K^H$ for some $H\supset G$.
  Therefore, $\bar\rho_K$ is injective and $\bar \rho_K \times \bar r_K$ is an isomorphism.

  To show that $\gamma_K \times_{I_K} g_K$ is the inverse map, it is enough to show that $\bar\rho_K\times \bar r_K \circ \gamma_K = \rho_K$ and $\bar\rho_K\times \bar r_K \circ g_K = r_K$.
  This is because $\rho_K\times_{I_K} r_K = \id_{\cNO(X, I_K)}$.
  We calculate
  \[
    \bar\rho_K\times \bar r_K(\gamma_K(a)) = \bar\rho_K\times \bar r_K(\omega_{\bX_K}(\Delta_K(a))) = \bar\rho_K(\Delta_K(a)) = \rho_K(a)
  \]
  for all $a\in A$.
  The last equality follows from Lemma \ref{l:bar_rho_K}.
  The equality for $g_K$ is proved similarly.
  The commutativity of the diagram is straightforward from the definitions.
\end{proof}

Compare the statements of the following corollary with Proposition \ref{p:giut}.
\begin{corollary}[$I$-relative GIUT]
  Let $(A, X)$ be a proper product system and let $I\subset A$ be a T-family.
  Suppose that $(\sigma, s)$ is an $I$-Cuntz-Nica-Pimsner covariant representation on the \Cs{} $D$.
  The map $\sigma\times_I s\colon \cNO(X,I)\to D$ is faithful if and only if $\sigma(a)Q^F_s = 0$ implies $a\in I^F$ and $(\sigma, s)$ admits a gauge action.
\end{corollary}
\begin{proof}
  By Proposition \ref{p:iso_of_relative}, the map $\sigma \times_I s\colon \cNO(X,I) \to D$ is equivalent to the CNP-representation $(\bar\sigma, \bar s) = ((\sigma\times_I s) \circ \bar \rho_k, (\sigma\times_I s)\circ \bar r_K)$ of $(\bA_K, \bX_K)$, where $K = K_I$.
  This reduces the statement to the ordinary GIUT, stated in Proposition \ref{p:giut}.
\end{proof}

We are now ready for the main result of the paper.
\begin{theorem} \label{t:classification-of-ideals}
  Let $(A, X)$ be a proper product system.
  The map $I\mapsto \cC_I$ defines a lattice isomorphism between:
  \begin{enumerate}
  \item the set of T-families and the set of gauge-invariant ideals of $\cNT(X)$;
  \item the set of O-families and the set of gauge-invariant ideals of $\cNO(X)$.
  \end{enumerate}
\end{theorem}
\begin{proof}
  We know that $\cNT(X) = \cNO(X, 0)$.
  Proposition \ref{p:iso_of_relative} shows that it is then isomorphic to $\cNO(\bX_{K_0})$, where $K_0$ is the smallest invariant family.
  By Proposition \ref{p:invariant_ideals_of_AK} we know that gauge-invariant ideals of $\cNO(\bX_{K_0})$ are in bijection with T-families.
  If $K$ is an invariant family, then the corresponding gauge-invariant ideal is the kernel of $\cNO(\bX_{K_0})\twoheadrightarrow \cNO(\bX_K)$.
  By Proposition \ref{p:iso_of_relative}, this surjection fits into the commutative diagram
  \[
    \begin{tikzcd}
      \cNO(\bX_{K_0}) \arrow[r, "{\bar \rho_{K_0} \times \bar r_{K_0}}"] \arrow[d, two heads] & {\cNT(X)} \arrow[d, two heads, "{\quot[\cC_{I_K}]}"] \\
      \cNO(\bX_{K}) \arrow[r, "{\bar \rho_{K} \times \bar r_{K}}"]                   & {\cNO(X, I_{K})}
    \end{tikzcd}.
  \]
  This shows that $K\mapsto C_{I_K}$ is a lattice isomorphism between invariant families and gauge-invariant ideals of $\cNT(X)$.
  Since $K\mapsto I_K$ is a lattice isomorphism between invariant families and T-families by Proposition \ref{p:T_and_inv_families_equiv}, we get the first part of the theorem.

  The second part follows from the fact that $\cC_I \subset \cC_{\cI}$ if and only if $\cI\preceq I$.
  This is exactly the condition that $I$ is an O-family.
\end{proof}

\section{Higher-rank graphs} \label{s:examples}
Higher-rank graphs and their \Cs{} were introduced by Kumjian and Pask in \cite{KP00}.
A graph of rank $n$ is a countable category $\Gamma$ of paths with a degree functor $\tup d\colon \Gamma \to \N^n$, which satisfies the factorization property: for any path $\gamma \in \Gamma$ and any element $\tup m \in \N^n$ with $\tup m \leq \tup d(\gamma)$, there are unique paths $\gamma(0, \tup m), \gamma(\tup m, \tup d(\gamma)) \in \Gamma$ such that $\tup d(\gamma(0, \tup m)) = \tup m$, $\tup d(\gamma(\tup m, \tup d(\gamma))) = \tup d(\gamma) - \tup m$, and $\gamma = \gamma(0, \tup m)\gamma(\tup m, \tup d(\gamma))$.
We denote by $s,r\colon \Gamma \to \Gamma^0$ the source and range maps, respectively.

We can associate a product system $(c_0(\Gamma^{0}), X(\Gamma))$ to an $n$-graph $\Gamma$ as follows.
The base algebra $c_0(\Gamma^{0})$ is just the algebra of functions that vanish on infinity on the discrete set of vertices $\Gamma^{0}$.
Analogously, $X(\Gamma^{\tup m})$ is the vector space of functions $x\colon \Gamma^{\tup m}\to \C$ on the set of paths of degree $\tup m$ such that $\sum_{\alpha\in \Gamma^{\tup m}, s(\alpha) = v} {|\xi(\alpha)|}^2 < \infty$ for all $v\in \Gamma^0$.
If $\alpha$ is a path of degree $\tup m$ and $v\in \Gamma^0$ is a vertex, then we denote by $\delta_\alpha \in X(\Gamma^{\tup m})$ and $\delta_v \in c_0(\Gamma^{0})$ the characteristic functions of $\alpha$ and $v$, respectively.

The multiplication $\mu^{\tup m, \tup k}$ is given by concatenation of paths:
\[
  \delta_\alpha \cdot \delta_\beta =
  \begin{cases}
    \delta_{\alpha\beta} & \text{if } s(\alpha) = r(\beta), \\
    0                    & \text{otherwise}
  \end{cases}
\]
for all paths $\alpha\in \Gamma^{\tup m}$ and $\beta\in \Gamma^{\tup k}$.
The action of $c_0(\Gamma^{0})$ is defined analogously and the scalar product on $X(\Gamma^{\tup m})$ is given by
\[
  \langle \delta_\alpha, \delta_\beta \rangle =
  \begin{cases}
    \delta_{s(\alpha)} & \text{if } \alpha = \beta, \\
    0                  & \text{otherwise}
  \end{cases}
\]
for all paths $\alpha, \beta\in \Gamma^{\tup m}$.

Raeburn and Sims \cite{RS05} introduced a class of higher-rank graphs called \emph{finitely aligned} graphs.
It was shown in \cite[Theorem 5.4]{RSY03} that the product system $(c_0(\Gamma^{0}), X(\Gamma))$ is compactly aligned if and only if $\Gamma$ is finitely aligned.
In this case, the algebra $C^*(\Gamma)$ of a higher-rank graph $\Gamma$ can be defined as the Cuntz-Nica-Pimsner algebra $\cNO(X(\Gamma))$.
Sims and Yeend showed in \cite[Propsition 5.4]{SY10} that this definition is equivalent to the earlier definition in terms of Cuntz-Krieger families.
Dor-On and Kakariadis defined a subclass of \emph{strongly finitely aligned} graphs in \cite[Definition 7.2]{DK2021}.
They proved that the product system is strongly compactly aligned if and only if $\Gamma$ is strongly finitely aligned.
This fact can be used to describe the higher-rank graph \Cs{} using simpler covariance conditions (see \cite[Theorem 7.6]{DK2021}).

Finally, a graph $\Gamma$ is called \emph{row-finite} if for every vertex $v\in \Gamma^0$ there are only finitely many paths with range $v$ of any given degree.
It is straightforward to show that the graph is row-finite if and only if the associated product system is proper.

Sims classified gauge-invariant ideals of $C^*(\Gamma)$ for finitely aligned graphs in \cite{S06}.
We will show how our results can be used to recover this classification in the case of row-finite graphs.
Moreover, we will give an alternative descriptions of the ideal lattices.

From now on, assume that $\Gamma$ is a row-finite graph.
If $V\subset \Gamma^0$ is a subset, then we denote by $c_0(V)$ the closed ideal generated by $\{\delta_v \colon v\in V\}$.
This defines a bijection between subsets of $\Gamma^0$ and ideals of $c_0(\Gamma^0)$.

For a subset $V\in \Gamma^0$ and $\tup m \in \N^n$, we define
\[
  (\Gamma^{\tup m})^{-1}(V) = \{v \in \Gamma^0 \colon \forall \alpha \in \Gamma^{\tup m} \text{ with } r(\alpha) = v \text{ we have } s(\alpha) \in V\}.
\]
In particular, $(\Gamma^{\tup m})^{-1}(\emptyset)$ is the set of $\tup m$-sources, i.e., the set of vertices not receiving any path of degree $\tup m$.
The following equalities are trivial:
\begin{equation} \label{eq:subset-ideals}
  \begin{split}
    (X(\Gamma^{\tup m}))^{-1}(c_0(V)) &= c_0((\Gamma^{\tup m})^{-1}(V)), \\
    \ker \varphi^{\tup m}_{X(\Gamma)} &= c_0((\Gamma^{\tup m})^{-1}(\emptyset)), \\
    (c_0(V))^\perp &= c_0(\Gamma^0 \setminus V).
  \end{split}
\end{equation}

We can use these formulas to describe the (pre-)CNP ideals $\cJ^F$ and $\cI^F$ of $(c_0(\Gamma^0), X(\Gamma))$.
For $F \in \cF(\Gamma)$, define the subset $\cW^F$ of $\Gamma^0$ by
\begin{align*}
  \cW^F & = \Gamma^0 \setminus \bigcap_{i \in F} (\Gamma^{\tup 1_i})^{-1}(\emptyset) = \\
        & = \{ v \in \Gamma^0 \colon v \text{ receives at least one edge of degree } \tup 1_i \text{ for some } i\in F\}.
\end{align*}
It is easy to see that $\cJ^F = c_0(\cW^F)$.
We can further define subsets $\cU^F \subset \Gamma^0$ by
\begin{align*}
  \cU^F & = \cW^F \cap \bigcap_{i \notin F} (\Gamma^{\tup 1_i})^{-1}(W^F) = \\
        & = \{ v \in \cW^F \colon \forall \gamma\in v\Gamma^{\tup m} \text{ with } \tup m \perp \tup 1_F \text{ we have } s(\gamma) \in W^F\} = \\
        & = \{ v \in \Gamma^0 \colon \forall \gamma \in v\Gamma^{\tup m} \text{ with } \tup m \perp \tup 1_F,~|s(\gamma)\Gamma^{\tup 1_i}|\neq 0 \text{ for some } i \in F\}.
\end{align*}
Here, the notation $v\Gamma^{\tup m}$ means the set of paths $\gamma$ with $r(\gamma) = v$ and $\tup d(\gamma) = \tup m\in \N^n$.
Dor-On and Kakariadis defined these sets in \cite[Definition 7.5]{DK2021} and called them sets of \emph{$F$-tracing vertices}.
With formulas \eqref{eq:subset-ideals} in mind, it is straightforward that $\cI^F = c_0(\cU^F)$.

We now want to describe $X(\Gamma)$-invariant ideals of $C^*(\Gamma)$.
For this, we define two properties of subsets of $\Gamma^0$.
\begin{definition} 
  Let $\Gamma$ be a row-finite $n$-graph and $V\subset \Gamma^0$ be a subset.
  \begin{enumerate}
  \item We say that $V$ is \emph{hereditary} if every path with range in $V$ has source in $V$.
    More precisely, $V$ is hereditary if and only if $V \subset (\Gamma^{\tup m})^{-1}(V)$ for all $\tup m \in \N^n$.
  \item We say that $V$ is \emph{$\cF$-saturated} if for every $F$-tracing vertex $v\in V$ such that for all $i\in F$ and all $\gamma \in v\Gamma^{\tup 1_i}$ we have $s(\gamma) \in V$, we have $v\in V$.
  \end{enumerate}
\end{definition}
Our definition of a hereditary set coincides with the definition of Sims in \cite[Definition 3.1]{S06}.
However, the notion of $\cF$-saturated sets is different from saturated sets of Sims.
We do not know if these two properties are equivalent.
The following result can shed some light on it.
\begin{theorem}\label{t:inv-hered-sat}
  Let $\Gamma$ be a row-finite $n$-graph and let $V\subset \Gamma^0$ be a subset.
  Then $c_0(V)$ is positively $X(\Gamma)$-invariant if and only if $V$ is hereditary and negatively invariant if and only if it is $\cF$-saturated.
  Therefore, $c_0(V)$ is $X(\Gamma)$-invariant if and only if $V$ is both hereditary and $\cF$-saturated.
\end{theorem}
\begin{proof}
  By Definition \ref{d:pos-inv}, the ideal $c_0(V)$ is positively $X(\Gamma)$-invariant if and only if $c_0(V) \subset (X(\Gamma^{\tup m}))^{-1}(c_0(V))$ for all $\tup m \in \N^n$.
  By \eqref{eq:subset-ideals}, this is equivalent to $V \subset (\Gamma^{\tup m})^{-1}(V)$ for all $\tup m \in \N^n$, which is the definition of a hereditary subset.

  The ideal $c_0(V)$ is negatively $X(\Gamma)$-invariant if and only if $\bigcap_{i\in F} (X(\Gamma)^{\tup 1_i})^{-1}(c_0(V)) \cap \cI^F \subset c_0(V)$ for all $F\in \cF$.
  By \eqref{eq:subset-ideals}, this is equivalent to $\bigcap_{i\in F} (\Gamma^{\tup 1_i})^{-1}(V) \cap \cU^F \subset V$ for all $F\in \cF$.
  A vertex $v$ is in $\bigcap_{i\in F} (\Gamma^{\tup 1_i})^{-1}(V) \cap \cU^F$ if and only if it is $F$-tracing and for all $i\in F$ and all $\gamma \in v\Gamma^{\tup 1_i}$ we have $s(\gamma) \in V$.
  Therefore, $c_0(V)$ is negatively invariant if and only if every such vertex is in $V$, which is the definition of an $\cF$-saturated subset.

  The last statement follows from the fact that $c_0(V)$ is $X(\Gamma)$-invariant if and only if it is both positively and negatively invariant by Theorem \ref{t:invariant_if_pos_and_neg}.
\end{proof}
\begin{corollary} \label{c:compare-sims}
  A hereditary subset $V\subset \Gamma^0$ of row-finite $n$-graph $\Gamma$ is $\cF$-saturated if and only if it is saturated in the sense of Sims (see \cite[Definition 3.1]{S06}).
\end{corollary}
\begin{proof}
  Section 3 of \cite{S06} shows that $c_0(V)$ is invariant if and only if $V$ is hereditary and saturated.
  On the other hand, we have just shown that $c_0(V)$ is invariant if and only if $V$ is hereditary and $\cF$-saturated.
  We conclude that these two properties are equivalent, whenever $V$ is hereditary.
\end{proof}
\begin{remark}
  Theorem \ref{t:inv-hered-sat} and Corollary \ref{c:compare-sims} only require results from Section \ref{s:invariant-ideals}.
  We have only assumed there that the product systems are strongly compactly aligned but not necessarily proper.
  Therefore, it is possible to extend these results to the case of strongly finitely aligned graphs.
\end{remark}

We now want to use the results of Section \ref{s:gauge-invariant-ideals} to describe ideals in $C^*(\Gamma) = \cNO(X(\Gamma))$ and $\cT(\Gamma)\coloneqq\cNT(X(\Gamma))$.
\begin{definition} \label{d:vertex-families}
  A collection $V = \{V^F\}_{F\in \cF}$ of subsets of $\Gamma^0$ is called a \emph{T-family of vertices} if the following condition holds: for every $F\in \cF$ and $i\in [n]\setminus F$, a vertex $v\in V^{F \cup \{i\}}$ is in $V^F$ if and only if for all $\gamma \in v\Gamma^{\tup 1_i}$ we have $s(\gamma) \in V^F$.
  It is further called an \emph{O-family of vertices} if $\cU^F \subset V^F$ for all $F\in \cF$.

  A collection $W = \{W^F\}_{F\in \cF} \subset \Gamma^0$ of vertices is called an \emph{invariant family} if $W^G = (\Gamma^{\tupi})^{-1}(W^G \cap W^{G\cup \{i\}})$ for all $G\in \cF$ and $i\in [n]\setminus G$.
\end{definition}
We can rewrite Definition \ref{d:vertex-families} as follows: $V=\{V^F\}_{F\in \cF}$ is a T-family of vertices if and only if $(\Gamma^{\tupi})^{-1}(V^F) \cap V^{F\cup \{i\}} = V^F$ for all $F\in \cF$ and $i\in [n]\setminus F$.
This is equivalent to $(X(\Gamma)^{\tupi})^{-1}(c_0(V^F)) \cap c_0(V^{F\cup \{i\}}) = c_0(V^F)$ for all $F\in \cF$ and $i\in [n]\setminus F$ by \eqref{eq:subset-ideals}, which is the definition of a T-family of ideals.
A T-family of ideals is an O-family of ideals if $\cI^F \subset c_0(V^F)$ for all $F\in \cF$.
This is the same as $\cU^F \subset V^F$ for all $F\in \cF$, which is the definition of an O-family of vertices.

Therefore, a collection of vertices is a T-family (resp. O-family) if and only if the corresponding collection of ideals is a T-family (resp. O-family).
It is also obvious that $V$ is an invariant family of vertices if and only if $c_0(W)$ is an invariant family of ideals.
Moreover, by Proposition \ref{p:T_and_inv_families_equiv}, there is an inclusion-preserving bijection between T-families and invariant families of vertices given by 
\[
  W_V^F \coloneqq (\Gamma^{\tup 1 - \tup 1_F})^{-1}(V^F).
\]

Let $W$ be an invariant family of vertices. 
We construct a higher-rank graph $\bGamma_W$ as follows.
Let $\bGamma_W^0 \coloneqq \bigsqcup_{F\in \cF} \Gamma^0\setminus W^F$.
For a vertex $v$ not in $W^F$, we denote by $v^F\in \bGamma_W^0$ the corresponding vertex in the $F$-component.
Furthermore, we define 
\[
  \bGamma_W^{\tup m} \coloneqq \bigsqcup_{F\in \cF} \Gamma^{\tup m}\setminus (\Gamma^{\tup m} W^F) = \bigsqcup_{F\in \cF} \{ \gamma \in \Gamma^{\tup m} \mid s(\gamma) \notin W^F \}.    
\]v
Analogously, we write $\gamma^F$ to denote the vertex in the $F$-component of $\bGamma_W^{\tup m}$ corresponding to $\gamma \in \Gamma^{\tup m}$.
Finally, we set $s(\gamma^F) = s(\gamma)^F$ and $r(\gamma^F) = r(\gamma)^{F\setminus \supp \tup m}$.
With the obvious path composition map, this defines a higher-rank graph $\bGamma_W$. 

\begin{theorem} \label{t:graph-classification}
  Let $\Gamma$ be a row-finite higher-rank graph.
  There is a lattice isomorphism between T-families (resp. O-families) of vertices and gauge-invariant ideals in $\cT(\Gamma)$ (resp. $C^*(\Gamma)$).

  Moreover, the quotient of $\cT(\Gamma)$ by the ideal corresponding to $V$ is isomorphic to $\cT(\bGamma_{W_V})$.
  Consequently, the algebra $\cT(\Gamma)$ as well as all it gauge-invariant quotients are higher-rank graph algebras.
\end{theorem}
\begin{proof}
  We have established above that T-families and O-families of vertices correspond bijectively to T-families and O-families of ideals.
  Therefore, the first claim follows immediately from Theorem \ref{t:classification-of-ideals}.

  For the second claim, it is easy to see that the product system $(c_0(\bGamma_{W_V}), X(\bGamma_{W_V}))$ is isomorphic to the product system $(\boldsymbol{c_0}(\Gamma)_{c_0(W_V)}, \boldsymbol{X}(\Gamma)_{c_0(W_V)})$ constructed in Proposition \ref{p:extended_prod_sys}.
  Then, the claim follows from Proposition \ref{p:iso_of_relative}.
\end{proof}
The second claim of Theorem \ref{t:graph-classification} is a generalization of \cite[Corollary 3.5]{BHRS02}, where the authors described quotients of rank-1 graph algebras as graph algebras of extended graphs.
To our knowledge, in case of higher-rank graphs, such extended graph was constructed only for the Toeplitz algebra by Pangalela in \cite{Pangalela16} but not for their quotients.

\printbibliography

\end{document}